\documentclass[11pt, reqno]{amsart}
\usepackage{amsfonts, amssymb}
\usepackage{graphicx}
\usepackage{hyperref}
\usepackage{parskip}
\usepackage{slashed}
\usepackage{fullpage}

\numberwithin{equation}{section}
\newtheorem{theorem}{Theorem}

\newtheorem{thm}{Theorem}[section]
\newtheorem{cor}[thm]{Corollary}
\newtheorem{lem}[thm]{Lemma}

\newtheorem{prop}[thm]{Proposition}

\theoremstyle{remark}
\newtheorem{rem}[thm]{Remark}

\theoremstyle{definition}
\newtheorem{defin}[thm]{Definition}

\renewcommand{\Re}{\mathrm{Re}}
\renewcommand{\Im}{\mathrm{Im}}

\newcommand{\e}{\mathbf{e}}

\newcommand{\R}{\mathbb{R}}
\newcommand{\Z}{\mathbb{Z}}

\newcommand{\C}{\mathbb{C}}

\newcommand{\cP}{\mathcal{P}}
\newcommand{\cQ}{\mathcal{Q}}

\newcommand{\mhi}{\mu_{\mathrm{hi}}}
\newcommand{\mlo}{\mu_{\mathrm{lo}}}
\newcommand{\nhi}{\nu_{\mathrm{hi}}}
\newcommand{\nlo}{\nu_{\mathrm{lo}}}

\newcommand{\la}{\langle}
\newcommand{\ra}{\rangle}
\newcommand{\Sp}{\mathbb{S}}

\newcommand{\ls}{\lesssim}
\newcommand{\due}{|D|^{-\frac{1}{2}}U^2_\e}
\newcommand{\dve}{|D|^{-\frac{1}{2}}V^2_\e}
\newcommand{\qo}{Q^0}
\newcommand{\snabla}{\slashed \nabla}

\DeclareMathOperator{\dist}{dist}
\renewcommand{\#}{\sharp}

\title{Local wellposedness of Chern-Simons-Schr\"odinger}

\author[B. Liu]{Baoping Liu}
\address{University of Chicago}
\email{baoping@math.uchicago.edu}
\author[P. Smith]{Paul Smith}
\address{University of California, Berkeley}
\email{smith@math.berkeley.edu}
\author[D. Tataru]{Daniel Tataru}
\address{University of California, Berkeley}
\email{tataru@math.berkeley.edu}

\thanks{The second author was supported by NSF grant DMS-1103877.
The third author was supported by NSF grant
DMS-0801261 and by the Simons Foundation. }

\begin{document}

\begin{abstract}
  In this article we consider the initial value problem for the
  Chern-Simons-Schr\"odinger model in two space dimensions. This is a
  covariant NLS type problem which is $L^2$ critical.  For this
  equation we introduce a so-called heat gauge, and prove that, with
  respect to this gauge, the problem is locally well-posed for initial
  data which is small in $H^s$, $s > 0$.
\end{abstract}

\maketitle

\tableofcontents

%%%%%%%%%%%%%%%%%%%%%%%%%%%%%%%%%%%%%%%%%%%%%%%%%%%%%%%
%%%%%%%%%%                                                     
%%%%%%%%%%                      Introduction                    
%%%%%%%%%%                                                    
%%%%%%%%%%%%%%%%%%%%%%%%%%%%%%%%%%%%%%%%%%%%%%%%%%%%%%%
\section{Introduction} \label{S:Introduction}
 The two dimensional Chern-Simons-Schr\"odinger system is a nonrelativistic quantum
model describing the dynamics of a large number of particles in the plane, which
interact both directly and via a self-generated electromagnetic field. The variables
we use to describe the dynamics are the scalar field $\phi$ describing the particle
system, and the electromagnetic potential $A$, which can be viewed as a one-form
on $\R^{2+1}$.  The associated covariant differentiation operators are
 defined in terms of the electromagnetic potential $A$ as
\begin{equation}
D_\alpha := \partial_\alpha+ i A_\alpha.
\label{D alpha}
\end{equation}
With this notation, the Lagrangian for this system
is
\begin{equation}
L(A,\phi) = \frac12 \int_{\R^{2+1}}  
\Im (\bar \phi D_t \phi) + |D_x \phi|^2 -\frac{g}2 |\phi|^4 dx dt +
\frac12 \int_{\R^{2+1}} A \wedge dA
\end{equation}
Although the electromagnetic potential $A$ appears explicitly in the Lagrangian,
it is easy to see that  locally  $L(A,\phi)$ only depends upon the electromagnetic field
$F = dA$. Precisely, the Lagrangian is invariant with respect to the transformations
\begin{equation}\label{gauge-freedom}
\phi \mapsto e^{-i \theta} \phi
\quad \quad
A \mapsto A + d \theta
\end{equation}
for compactly supported real-valued functions $\theta(t, x)$.

Computing the Euler-Lagrange equations for the above Lagrangian,
one obtains a covariant NLS equation for $\phi$, coupled with
equations giving the electromagnetic field in terms of $\phi$, as follows:
\begin{equation}
\begin{cases}
D_t \phi
&=
i D_\ell D_\ell \phi + i g \lvert \phi \rvert^2 \phi
\\
\partial_t A_1 - \partial_1 A_t
&= - J_2
\\
\partial_t A_2 - \partial_2 A_t
&=
J_1
\\
\partial_1 A_2 - \partial_2 A_1
&=
-\frac{1}{2} \lvert \phi \rvert^2
\end{cases}
\label{CSS}
\end{equation}
where we  use  $J_i$ to denote
\begin{equation*}
J_i := \Im(\bar{\phi} D_i \phi)
\end{equation*}
Regarding indices, we use $\alpha = 0$ for the time variable $t$
and $\alpha = 1, 2$ for the spatial variables $x_1, x_2$. When we
wish to exclude the time variable in a certain expression, we
switch from Greek indices to Roman. Repeated indices are assumed
to be summed. We discuss initial conditions in \S \ref{S:Gauge}.

The system (\ref{CSS}) is a basic model of Chern-Simons dynamics
\cite{JaPi91, EzHoIw91, EzHoIw91b, JaPi91b}. For further physical
motivation for studying (\ref{CSS}), see \cite{JaTe81, MaPaSo91,
Wi90}.

The system (\ref{CSS}) has the gauge invariance \eqref{gauge-freedom}.
It is also Galilean-invariant and has conserved
\emph{charge}
\begin{equation*}
M(\phi) := \int_{\R^2} \lvert \phi \rvert^2 dx
\end{equation*}
and \emph{energy}
\begin{equation*}
E(\phi) :=
\frac12 \int_{\R^2}
\lvert D_x \phi \rvert^2
- \frac{g}{2} \lvert \phi \rvert^4 dx.
\end{equation*}
As the scaling symmetry
\[
\phi(t,x)\rightarrow \lambda\phi(\lambda^2t, \lambda x ), \quad
\phi_0(x)\rightarrow \lambda\phi_0(\lambda x); \quad \lambda>0,
\]
preserves the charge of the initial data $M(\phi_0)$, $L^2_x$ is
the critical space for equation (\ref{CSS}).

Local wellposedness in $H^2$ is established in \cite{BeBoSa95}.  Also
given are conditions ensuring finite-time blowup. With a
regularization argument, \cite{BeBoSa95} demonstrates global existence
(but not uniqueness) in $H^1$ for small $L^2$ data.

Our goal in this paper is to establish local wellposedness for
(\ref{CSS}) in spaces over the full subcritical range $H^s$ with $s > 0$.
However, in order to state the result we need to first remove
the gauge freedom by choosing a suitable gauge. This is done in the next section,
which ends with our main result.

%%%%%%%%%%%%%%%%%%%%%%%%%%%%%%%%%%%%%%%%%%%%%%%%%%%%%%%
%%%%%%%%%%                                                      
%%%%%%%%%%                      Gauge selection                
%%%%%%%%%%                                                     
%%%%%%%%%%%%%%%%%%%%%%%%%%%%%%%%%%%%%%%%%%%%%%%%%%%%%%%
\section{Gauge selection} \label{S:Gauge}
In order to interpret the  Chern-Simons-Schr\"odinger system \eqref{CSS}
as a well-defined time evolution, we need to impose a suitable gauge
condition which eliminates the  gauge freedom described in \eqref{gauge-freedom}.

One can relate the gauge fixing problem here to the similar difficulty
occurring in the study of wave and Schr\"odinger maps.
Both wave maps and Schr\"odinger maps are geometric evolution equations,
and in such settings the function $\phi$ takes values not in $\C$,
but rather more generally in some (suitable) manifold $M$.
A gauged system arises when
considering evolution equations at the level of the (pullback of the) tangent bundle
$\phi^* TM$, where $\phi^*$ denotes the pullback.

A classical gauge choice is the \emph{Coulomb gauge}, which is derived by imposing
the constraint $\nabla \cdot A_x = 0$. In low dimension, however, the Coulomb gauge
has unfavorable $\mathrm{high} \times \mathrm{high} \to \mathrm{low}$ interactions.
To overcome this difficulty in the $d = 2$ setting of wave maps into hyperbolic space,
Tao \cite{Tao04} introduced the caloric gauge as an alternative. See \cite{Tao08}
for an application of the caloric gauge to large data wave maps in $d=2$ and 
\cite{BeIoKeTa11} for an application to small data Schr\"odinger maps in $d = 2$.
We refer the reader to \cite[Chapter 6]{Tao06} for a lengthier discussion and a comparison
of various gauges.

Unfortunately, the direct analogue of the caloric gauge
for the Chern-Simons-Schr\"odinger system does not result in any
improvement over the Coulomb gauge.
Instead, in this
article we adopt from \cite{De07} a different variation of the Coulomb
gauge called the \emph{parabolic gauge}. We shall also refer to it as
the \emph{heat gauge}.  The defining condition of the heat gauge is
\begin{equation}
\nabla \cdot A_x = A_t
\label{heat gauge condition}
\end{equation}
Differentiating in the $x_1$ and $x_2$ directions the second
and third equations (respectively) in (\ref{CSS}) yields
\begin{equation*}
\begin{cases}
\partial_t \partial_1 A_1 - \partial_1^2 A_t
&=
- \partial_1 \Im(\bar{\phi} D_2 \phi) \\
\partial_t \partial_2 A_2 - \partial_2^2 A_t
&=
\partial_2 \Im(\bar{\phi} D_1 \phi)
\end{cases}
\end{equation*}
Adding these, we get
\begin{equation*}
\partial_t (\nabla \cdot A_x) - \Delta A_t = - \partial_1 J_2 + \partial_2 J_1,
\end{equation*}
which, in view of the heat gauge condition (\ref{heat gauge condition}),
implies that $A_t$ evolves according to the nonlinear heat equation
\begin{equation}
(\partial_t - \Delta) A_t = - \partial_1 J_2 + \partial_2 J_1
\label{At evo}
\end{equation}
Similarly, we obtain (coupled) parabolic evolution equations for $A_1$ and $A_2$:
\begin{equation}
\begin{cases}
(\partial_t - \Delta) A_1
&= -J_2 - \frac{1}{2} \partial_2 \lvert \phi \rvert^2
\\
(\partial_t - \Delta) A_2
&= J_1 + \frac{1}{2} \partial_1 \lvert \phi \rvert^2
\end{cases}
\label{Ax evo}
\end{equation}
We still retain the freedom to impose initial conditions for the
parabolic equations for $A$ in (\ref{At evo}) and (\ref{Ax evo}), in
any way that is consistent with the last equation in \eqref{CSS}.
We impose
\[
A_t(0) = \nabla \cdot
A_x(0) = 0
\]
 To see that such a choice is consistent with
(\ref{CSS}), observe that $\nabla \cdot A_x(0) = 0$ coupled with
the fourth equation of (\ref{CSS}) yields the system
\begin{equation}
\begin{cases}
\partial_1 A_1(t=0) + \partial_2 A_2(t=0) = 0 \\
\partial_1 A_2(t=0) - \partial_2 A_1(t=0) = - \frac{1}{2} \lvert \phi_0 \rvert^2,
\end{cases}
\label{Ax initial system}
\end{equation}
which in turn implies
\begin{equation}
\begin{cases}
\Delta A_1(t = 0) &= \frac{1}{2} \partial_2 \lvert \phi_0 \rvert^2 \\
\Delta A_2(t = 0) &= -\frac{1}{2} \partial_1 \lvert \phi_0 \rvert^2
\label{Ax initial data}
\end{cases}
\end{equation}
Substituting (\ref{Ax initial data}) into (\ref{Ax evo}) yields
\begin{equation*}
\begin{cases}
\partial_t A_1(t = 0) &= -\Im(\bar{\phi} D_2 \phi) \\
\partial_t A_2(t = 0) &= \Im(\bar{\phi} D_1 \phi),
\end{cases}
\end{equation*}
which is exactly what we obtain directly from the second and third
equations of (\ref{CSS}) at $t = 0$ with the choice $A_t(t = 0) \equiv 0$.

So having imposed an additional equation in order to fix a gauge,
we study the initial value problem for the system
\begin{equation}
\begin{cases}
D_t \phi
&=
i D_\ell D_\ell \phi + i g \lvert \phi \rvert^2 \phi
\\
\partial_t A_1 - \partial_1 A_t
&=
- \Im(\bar{\phi} D_2 \phi)
\\
\partial_t A_2 - \partial_2 A_t
&=
\Im(\bar{\phi} D_1 \phi)
\\
\partial_1 A_2 - \partial_2 A_1
&=
-\frac{1}{2} \lvert \phi \rvert^2
\\
A_t
&=
\nabla \cdot A_x
\end{cases}
\label{CSSwithGauge}
\end{equation}
with initial data
\begin{equation}
\begin{cases}
\phi(0, x) &= \phi_0(x) \\
A_t(0, x) &= 0 \\
A_1(0, x) &= \frac{1}{2}\Delta^{-1}\partial_2 \lvert \phi_0 \rvert^2(x) \\
A_2(0, x) &= - \frac{1}{2} \Delta^{-1} \partial_1 \lvert \phi_0 \rvert^2(x)
\end{cases}\label{CSSdata}
\end{equation}
Our main result is the following.
\begin{thm}
  For any small initial data $\phi_0\in H^s{(\R^2)}$, $s>0$, the
  equation (\ref{CSSwithGauge}) with initial data (\ref{CSSdata}) has
  solution $\phi(t,x)\in C([0,1],H^s(\R^2))$, which is the unique
  uniform limit of smooth solutions. In addition, $\phi_0 \mapsto
  \phi$ is Lipschitz continuous from $H^s(\R^2)$ to
  $C([0,1],H^s(\R^2))$.
\end{thm}

We remark that ideally one would like to have global well-posedness
for small data in $L^2$. Unfortunately, in our arguments we encounter
logarithmic divergencies at nearly every step with respect to the
$L^2$ setting, making it impossible to achieve this goal.

Another interesting remark is that while the initial system \eqref{CSS}
is time reversible, the parabolic evolutions added by our gauge choice remove the
time reversibility. One may possibly view this as a disadvantage
of our gauge choice.

Our result is proved via a fixed point argument in a topology $X^{s}$, defined later, which is
stronger than the $C([0,1],H^s(\R^2))$ topology. Thus we directly obtain uniqueness
in $X^s$, as well as Lipschitz dependence on the initial data with respect to the $X^s$
topology.

%%%%%%%%%%%%%%%%%%%%%%%%%%%%%%%%%%%%%%%%%%%%%%%%%%%%%%%
%%%%%%%%%%                                                   
%%%%%%%%%%              Reductions using the heat gauge             
%%%%%%%%%%                                              
%%%%%%%%%%%%%%%%%%%%%%%%%%%%%%%%%%%%%%%%%%%%%%%%%%%%%%%
\section{Reductions using the heat gauge} \label{reductionEq}
Let $\tilde{f}$ denote the space-time Fourier transform
\[
\tilde{f}(\tau, \xi) := \iint e^{-i (t \tau  + x \cdot \xi)} f(t, x) dt dx
\]

We define $H^{-1}$ as the Fourier multiplier
\begin{equation}
H^{-1}f:=
\frac{1}{(2\pi)^3}
\int\frac{1}{i\tau+|\xi|^2}e^{i(t\tau +
x\cdot\xi)}\tilde{f}(\tau,\xi)d\tau d\xi
\end{equation}

Applied to initial data, it takes the form
\[
H^{-1}(f(x)\delta_{t=0})
=
\mathbf{1}_{\{ t\geq 0 \}}e^{t\Delta}f(x)
\]

We define $H^{-\frac{1}{2}}$ similarly:
\begin{equation}
H^{-\frac{1}{2}}f:=
\frac{1}{(2\pi)^3}
\int\frac{1}{(i\tau+|\xi|^2)^{\frac{1}{2}}}e^{i(t\tau
+ x\cdot\xi)}\tilde{f}(\tau,\xi)d\tau d\xi
\end{equation}
Here we use the principal square root of the complex-valued function
$i\tau+|\xi|^2$ by taking the positive real axis as the branch
cut.  As the above symbol is still holomorphic for $\tau$ in the lower
half-space, it follows that its kernel is also supported in $t \geq 0$.
In what follows all these operators are applied only to functions supported on positive time
intervals.

Using (\ref{At evo}), we can rewrite $A_t $ as
\begin{equation}\label{At}
A_t = - H^{-1}((Q_{12}(\bar{\phi}, \phi))) - H^{-1}(\partial_1(
A_2 \lvert \phi \rvert^2)) + H^{-1}(\partial_2( A_1 \lvert \phi
\rvert^2)),
\end{equation}
where
$Q_{12}(\phi, \bar{\phi}) := \Im(\partial_1\phi\partial_2\bar{\phi}-\partial_2\phi\partial_1\bar{\phi})$.

Similarly, by (\ref{Ax evo}) and initial condition (\ref{Ax
initial data}), we can rewrite $A_x$ as follows:
\begin{equation}\label{A1A2Initial}
\begin{split}
A_1 &= H^{-1}A_1( 0)-H^{-1} [\Re(\bar{\phi} \partial_2 \phi)+\Im(\bar{\phi} \partial_2\phi)]  - H^{-1}(A_2 \lvert \phi \rvert^2) \\
A_2 &=H^{-1}A_2( 0) +H^{-1} [\Re(\bar{\phi}
\partial_1 \phi) + \Im(\bar{\phi}
\partial_1 \phi) ] + H^{-1}(A_1 \lvert \phi \rvert^2)
\end{split}
\end{equation}
Here
\[
A_1( 0)=\frac{1}{2}\Delta^{-1}\partial_2\lvert \phi_0 \rvert^2
\quad \text{and} \quad
A_2( 0)=-\frac{1}{2}\Delta^{-1}\partial_1\lvert \phi_0 \rvert^2
\]
Our strategy will be to use the contraction principle in the equations \eqref{A1A2Initial}
in order to bound $A_1$ and $A_2$  in terms of $\phi$, and then to use
\eqref{At} to estimate $A_t$.  The contraction principle is not applied directly to $A_1$ and
$A_2$, but instead to
\[
B_1 =   H^{-1}(A_2 \lvert \phi \rvert^2), \qquad B_2 =  H^{-1}(A_1 \lvert \phi \rvert^2),
\]
These functions solve the system
\begin{equation}\label{B1B2}
\begin{split}
B_1 &=  H^{-1}( (H^{-1}A_2( 0) +H^{-1} [\Re(\bar{\phi}
\partial_1 \phi) + \Im(\bar{\phi}
\partial_1 \phi) ])|\phi|^2) + H^{-1} (B_2 |\phi|^2)
\\
B_2 & =  H^{-1}( (H^{-1}A_1( 0)-H^{-1} [\Re(\bar{\phi} \partial_2 \phi)+\Im(\bar{\phi} \partial_2\phi)] )|\phi|^2) -  H^{-1} (B_1 |\phi|^2)
\end{split}
\end{equation}

We observe here that the first components of $A_1$ and $A_2$ depend
only on the initial data; therefore they effectively act almost as
stationary electromagnetic potentials for the linear Schr\"odinger
equation. The difficulty is that even if $\phi_0$ is localized, both
of these components have only $|x|^{-1}$ decay at infinity, which in
general would make them nonperturbative long range potentials.
Fortunately $A_1(0)$ and $A_2(0)$ are not independent, and taking
into account their interrelation will allow us to still treat their effects
in the Schr\"odinger equation as perturbative. However, this ends up  causing
 considerable aggravation in the construction of our function spaces.

Now we turn our attention to the first equation in \eqref{CSS}, which
we expand using (\ref{D alpha}) as
\begin{equation}\label{phi evolution}
(i \partial_t + \Delta) \phi = N(\phi,A):=- 2 i A_\ell \partial_\ell \phi - i
\partial_\ell A_\ell \phi + A_t \phi + A_x^2 \phi - g \lvert \phi
\rvert^2 \phi
\end{equation}
where $A_t$, $A_1$ and $A_2$ are given by \eqref{At} and \eqref{A1A2Initial}.
Our plan is to solve this equation perturbatively. However, in order for
this to work, we need to use \eqref{At} and \eqref{A1A2Initial}
to expand the $A$'s in the nonlinearity and replace them by $B$'s.
Even this expansion is not sufficient in the case of the first term in $ N(\phi,A)$, for which
we need to expand once more. Eventually this leads to an expression of the form
\[
N(\phi,A) = L \phi + N_{3,1} + N_{3,2} + N_{3,3} + N_{5,1} + N_{5,2}
+ \sum_{j=0}^7 E_j
\]
where the terms above are as follows:
\begin{enumerate}
\item $L$ contains the linear terms in $\phi$, which arise from the
first term in $N$ and the first term in the expansion of $A$. It has the
form
\begin{equation}\label{L}
L \phi = i Q_{12} (C,\phi) ,
\qquad C = H^{-1} \Delta^{-1} \lvert \phi_0 \rvert^2
\end{equation}
where
\[
Q_{12} (C,\phi) = \partial_1 C \partial_2 \phi - \partial_2 C \partial_1 \phi
\]
As mentioned before, this term significantly affects our function spaces constructions.

\item The terms $N_{3,1} $, $N_{3,2}$ and $N_{3,3}$  are the cubic
terms in $\phi$, described as follows:
\begin{equation}\label{N31}
 N_{3,1} = H^{-1}(\bar{\phi}\partial_1\phi)\partial_2\phi-
H^{-1}(\bar{\phi}\partial_2\phi)\partial_1\phi,
\end{equation}
originates from the first term in $N$ and the second term in $A$.
The ``null'' structure in $N_{3,1}$ is crucial in our estimates.
\begin{equation}\label{N32}
 N_{3,2} = H^{-1}(Q_{12}(\bar{\phi}, \phi)) \phi,
\end{equation}
originates from the second term in $N$ and the second term in $A$,
and also from the third term in $A$ and the first term in $A_t$.
This also exhibits a null structure that we take advantage of.
\begin{equation}\label{N33}
 N_{3,3} = |\phi|^2 \phi
\end{equation}
is the contribution of the last term in $N$.

\item The terms $N_{5,1} $, $N_{5,2}$ and $N_{5,2}$  are the quintic
terms in $\phi$, described as follows:
\begin{equation}\label{N51}
 N_{5,1} =  H^{-1}( H^{-1}(\bar \phi \partial \phi) |\phi|^2) \partial \phi
\end{equation}
occurs in the reexpansion of $A_{1}$ and $A_2$ in the first term in $N$.
\begin{equation}\label{N52}
 N_{5,2} =  H^{-1}(\bar \phi \partial \phi) H^{-1}(\bar \phi \partial \phi) \phi
\end{equation}
occurs in the fourth term in $N$, corresponding to the second term
in $A_1, A_{2}$.
\begin{equation}\label{N53}
 N_{5,3} =   H^{-1}\partial ( H^{-1}(\bar \phi \partial \phi) |\phi|^2)  \phi
\end{equation}
occurs in the third term in $N$, corresponding to the
second term in $A_{1},A_2$ arising in the $A_t$ expression.

In all these terms it is neither important which spatial derivatives are applied nor
where the bar goes in $\bar \phi \partial \phi$. While they look somewhat different,
in the proofs of the multilinear estimates they turn out to be essentially equivalent by a
duality argument.

\item The ``error'' terms are those which can be estimated in a relatively
simpler manner. We begin with the multilinear terms containing the
data for $A_x$, namely
\begin{equation}\label{E1}
E_1 = H^{-1}( H^{-1} A_x(0) |\phi|^2) \partial \phi
\end{equation}
from the reexpansion of the first term in $N$,
\begin{equation}\label{E2}
E_2 = H^{-1}\partial ( H^{-1} A_x(0) |\phi|^2)  \phi
\end{equation}
from the  second and third terms in $N$,
\begin{equation}\label{E3}
E_3 =H^{-1} A_x(0) H^{-1}( \bar \phi \partial \phi) \phi
\end{equation}
from the fourth term in $N$,
\begin{equation}\label{E4}
E_4 =(H^{-1} A_x(0))^2  \phi +   H^{-1} A_x(0) B   \phi + B^2 \phi.
\end{equation}
Finally we conclude with the remaining terms involving $B$,
\begin{equation}\label{E5}
E_5 = H^{-1}(B |\phi|^2) \partial \phi
\end{equation}
from the reexpansion of the first term in $N$,
\begin{equation}\label{E6}
E_6 = H^{-1}\partial ( B |\phi|^2)  \phi
\end{equation}
from the  second and third term in $N$,
\begin{equation}\label{E7}
E_7 = H^{-1}( \bar \phi \partial \phi) B \phi
\end{equation}
from the fourth term in $N$.
\end{enumerate}

%%%%%%%%%%%%%%%%%%%%%%%%%%%%%%%%%%%%%%%%%%%%%%%%%%%%%%%
%%%%%%%%%%                                                     
%%%%%%%%%%                      Function spaces               
%%%%%%%%%%                                                     
%%%%%%%%%%%%%%%%%%%%%%%%%%%%%%%%%%%%%%%%%%%%%%%%%%%%%%%
\section{Function spaces}\label{sect:fs}

\noindent In this section we define function spaces as in
\cite{KoTa07, HeTaTz10}, but with some suitable adaptations to the
problem at hand.  Spaces similar to those in \cite{KoTa07, HeTaTz10}
have been used to obtain critical results in different problems
\cite{HHK09,HeTaTz11}. We refer the reader to \cite[\S 2]{HHK09} for
detailed proofs of the basic properties of $U^p, V^p$ spaces.

For a unit vector $\e \in \mathbb{S}^1$, we denote by
$H_{\e}$ its orthogonal complement in $\mathbb{R}^2$ with
the induced measure. Define the lateral spaces
$L^{p,q}_{\e}$ with norms
\[
\|f\|_{L^{p,q}_\e}
=
\left( \int_{\mathbb{R}} \left( \int_{H_\e \times \R }|f(x\e+x',t)|^{q}dx'dt \right)^{\frac{p}{q}}
dx\right)^{\frac{1}{p}}
\]
with the usual modifications when $p = \infty$ or $q = \infty$.

Define the operator $P_{N,\e}$ by the Fourier multiplier
$\xi\rightarrow \psi_{N}(\xi\cdot \e)$, where $\psi_N$
has symbol $\psi_N(\xi)$ given by (\ref{PNsymbol}).
The following smoothing estimate plays an important
role in our analysis.
\begin{lem}[Local smoothing \cite{IoKe06, IoKe07}]
Let $f\in L^2(\R^2)$, $N\in 2^{\Z}, N\geq 1$, and $\mathbf{e}\in
\mathbb{S}^1$. Then
\begin{equation}\label{est:localsmoothing}
\|e^{it\Delta}P_{N,\mathbf{e}}f\|_{L^{\infty,
2}_{\mathbf{e}}}\lesssim N^{-\frac{1}{2}}\|f\|_{L^2}
\end{equation}
\end{lem}

Also recall the well-known Strichartz estimates.
\begin{lem}[Strichartz estimates \cite{St77, Tao06}] Let $(q,r)$ be
  any admissible pair of exponents, i.e.
  $\frac{1}{q}+\frac{1}{r}=\frac{1}{2}$ and $(q,r)\neq(2,\infty)$.
  Then we have the homogeneous Strichartz estimate
\begin{equation}\label{est:strichartz}
\|e^{it\Delta}f\|_{L^q_tL^r_x(\mathbb{R}\times\mathbb{R}^2)}\lesssim
\|f\|_{L^2_{x}(\mathbb{R}^2)}
\end{equation}
\end{lem}

\subsection{{$U^p$} and {$V^p$} spaces} \label{sec:upvp}
Throughout this section let $H$ be a separable Hilbert space over
$\C$. Let $\mathcal{Z}$ be the set of finite partitions
$-\infty \leq t_0<t_1<\ldots<t_K\leq \infty$ of the real line. If
$t_K=\infty$ and $v : \R \to H$, then we adopt the convention that
$v(t_K):=0$. Let $\chi_I:\R\to\R$ denote the (sharp)
characteristic function of a set $I\subset \R$.

\begin{defin}\label{def:u}
  Let $1\leq p <\infty$. For any $\{t_k\}_{k=0}^K \in \mathcal{Z}$ and
  $\{\phi_k\}_{k=0}^{K-1} \subset H$ with
  $\sum_{k=0}^{K-1}\|\phi_k\|_{H}^p=1$, we call the
  function $a:\R \to H$ defined by
  \begin{equation*}
    a=\sum_{k=1}^K\chi_{[t_{k-1},t_k)}\phi_{k-1}
  \end{equation*}
  a $U^p$-atom. We define the atomic space $U^p(\R,H)$ as the set of all
  functions $u: \R \to H$ admitting a representation
  \begin{equation*}
    u=\sum_{j=1}^\infty \lambda_j a_j \; \text{ for } U^p\text{-atoms } a_j,\;
    \{\lambda_j\}\in \ell^1
  \end{equation*}
  and endow it with the norm
  \begin{equation}\label{eq:norm_u}
    \|u\|_{U^p}:=\inf \left\{\sum_{j=1}^\infty |\lambda_j|
      :\; u=\sum_{j=1}^\infty \lambda_j a_j,
      \,\lambda_j\in \C,\; a_j \text{ a $U^p$-atom}\right\}
  \end{equation}
\end{defin}

\begin{rem}\label{rmk:prop_up} The spaces $U^p(\R,H)$ are Banach spaces and
  we observe that $U^p(\R,H)\hookrightarrow L^\infty(\R;H)$. Every
  $u\in U^p(\R,H)$ is right-continuous. On occasion the space $U^p$ is defined
in a restricted fashion by requiring that $t_0 > -\infty$. Then $u$ tends to $0$ as $t \to
  -\infty$. The only difference between the two definitions is in whether or not one adds
constant functions to $U^p$.
\end{rem}

\begin{defin}\label{def:v}
  Let $1\leq p<\infty$ , We define $V^p(\R,H)$ as the space of all functions $v:\R\to
    H$ such that
    \begin{equation}\label{eq:norm_v}
      \|v\|_{V^p}
      :=\sup_{\{t_k\}_{k=0}^K \in \mathcal{Z}} \left(\sum_{k=1}^{K}
        \|v(t_{k})-v(t_{k-1})\|_{H}^p\right)^{\frac{1}{p}}
    \end{equation}
    is finite.  %\footnote{Here we make use of our convention       $v(\infty)=0$.}.

%  \begin{enumerate}
  %\item    
  %\item Likewise, let $V^p_{rc}(\R,H)$ denote the closed subspace of
    %all right-continuous functions $v:\R\to H$ such that $\lim_{t\to
      %-\infty}v(t)=0$.%, endowed with the same norm \eqref{eq:norm_v}.
  %\end{enumerate}
\end{defin}

\begin{rem}\label{rmk:prop_v}
We require that two $V^p$ functions be equal if they are equal in the sense of distributions. 
Since such functions have at most countably many discontinuous points in time, 
we adopt the convention that all $V^p$ functions are right continuous, i.e.,
we assume we always work with the unique right-continuous representative
from the equivalence class.

  The spaces $V^p(\R,H)$   are Banach spaces and
  satisfy
\begin{equation} \label{up-emb}
U^p(\R,H)\hookrightarrow V^p(\R,H)\hookrightarrow U^q(\R,H)
  \hookrightarrow
  L^\infty(\R;H), \qquad p < q
\end{equation}
%  The space $V^p_{rc}(\R,H)$ inherits its norm from $V^p(\R,H)$.
\end{rem}

We denote by $DU^p$ the space of distributional derivatives of $U^p$ functions.
Then we have the following  very useful duality property:
\begin{lem}\label{lem:dual}
The following duality holds
\begin{equation}\label{up-dual}
(DU^p)^* = V^{p'}, \qquad 1 \leq p < \infty
\end{equation}
with respect to a duality relation that extends the standard $L^2$ duality.
\end{lem}
We refer the reader to \cite{HHK09} for a more detailed discussion.

We also record a useful interpolation property of the spaces $U^p$
and $V^p$ (cf. \cite[Proposition 2.20]{HHK09}).
\begin{lem}\label{lem:interpol}
  Let $q_1,q_2>2$, $E$ be a Banach space and
  \[T:U^{q_1}\times U^{q_2}\to E\] a bounded
  bilinear operator with $\|T(u_1,u_2)\|_E \leq C
  \prod_{j=1}^2\|u_j\|_{U^{q_j}}$.  In addition, assume that there
  exists $C_2\in (0,C]$ such that the estimate $\|T(u_1,u_2)\|_E
  \leq C_2 \prod_{j=1}^2\|u_j\|_{U^2}$ holds true.  Then $T$
  satisfies the estimate
  \[
  \|T(u_1,u_2)\|_E \ls C_2
  (\ln\frac{C}{C_2}+1)^2\prod_{j=1}^2\|u_j\|_{V^2}, \quad u_j \in
  V^2,\; j=1,2.
  \]
\end{lem}
\begin{proof}
The proof is the same as that in \cite[Lemma 2.4]{HeTaTz11}.
For fixed $u_2$, let $T_1u:=T(u,u_2)$. Then we have that
\[
\|T_1u\|_E\leq D_1 \|u\|_{U^{q_1}}
\quad \text{and} \quad
\|T_1u\|_E\leq D_1' \|u\|_{U^{2}}.
\]
Here $D_1=C \|u_2\|_{U^{q_2}}, D_1'=C_2 \|u_2\|_{U^{2}} $.

From the fact that $\|u_2\|_{U^{q_j}}\leq \|u_2\|_{U^{2}}$ and
\cite[Proposition 2.20]{HHK09}, we obtain
\begin{equation}\label{eq: u1vu2u}
  \|T(u_1,u_2)\|_E=\|T_1u_1\|_E \lesssim C_2 (\ln \frac{C}{C_2} +1)
  \|u_1\|_{V^2}\|u_2\|_{U^{2}} 
  \end{equation}
  Then we can repeat the argument by fixing $u_1$, using
  estimate (\ref{eq: u1vu2u}),
  and
  \[\|T(u_1,u_2)\|_E \leq C
  \prod_{j=1}^2\|u_j\|_{U^{q_j}} \leq C
  \|u_1\|_{V^2}\|u_2\|_{U^{q_j}}\]
%  For fixed $u_2,u_3$ let $T_1u:=T(u,u_2,u_3)$. From the assumption we
%  have
%  \begin{align*}
%    \|T_1u\|_E&\leq D_1 \|u\|_{U^{q_1}} \text{ where }
%    D_1=C \|u_2\|_{U^{q_2}}\|u_3\|_{U^{q_3}},\\
%    \|T_1u\|_E&\leq D_1' \|u\|_{U^{2}} \text{ where } D_1'=C_2
%    \|u_2\|_{U^{2}}\|u_3\|_{U^{2}}.
%  \end{align*}
%  An application of \cite[Proposition 2.20]{HHK09} and the bound
%  $\|u_j\|_{U^{q_j}}\leq \|u_j\|_{U^{2}}$ for $j=2,3$ yield
%  \[
%  \|T_1u\|_E \ls C_2 (\ln \frac{C}{C_2} +1)
%  \|u\|_{V^2}\|u_2\|_{U^{2}}\|u_3\|_{U^{2}}.
%  \]
%  The embedding $V^2_{rc}\hookrightarrow U^{q_j}$ allows us to repeat
%  this argument with respect to the second and third argument of $T$.
\end{proof}

%We define the spatial Fourier coefficients
%\[
%\widehat{f}(\xi):=(2\pi)^{-3/2}\int_{ [0,2\pi]^{3}} e^{-ix\cdot
%\xi} f(x) \;dx,\; \xi \in \Z^{3}.
%\]
%and the space time Fourier transform
%\[
%\mathcal{F}u(\tau,\xi):=(2\pi)^{-2}\int_{ \R\times [0,2\pi]^{3}}
%e^{-i(x\cdot \xi+t\tau)} u(t,x) \;dt dx,\; (\tau,\xi) \in
%\R\times\Z^{3}.
%S\]

Let $\psi : \R \to [0,1]$ be a smooth even function compactly
supported in $[-2, 2]$ and equal to 1 on $[-1, 1]$. For dyadic
integers $N \geq 1 $ , set
\begin{equation} \label{PNsymbol}
\psi_N(\xi)=\psi\Big(\frac{|\xi|}{N}\Big)-\psi\Big(\frac{2|\xi|}{N}\Big),
\quad \text{for } N\geq 2 \quad \text{and} \quad
\psi_1(\xi)=\psi(|\xi|).
\end{equation}
For each such $N \geq 1$, define the frequency localization
operator $P_N:L^2(\R^2)\to L^2(\R^2)$ as the Fourier multiplier
with symbol $\psi_N$. Moreover, let $P_{\leq N}:=\sum_{1\leq M\leq
N}P_M$. We set $u_N:=P_Nu$ for short.
%$u_{\leq N} := P_{\leq N} u$

% Let $s \in \R$.  We define the Sobolev space $H^s(\R^2)$ as the
% space of all $L^2(\R^2)$-functions for which the norm
% \[
% \|f\|_{H^s(\R^2)}:=\left(\sum_{N \geq 1,
% \textnormal{ dyadic}}N^{2s} \|P_N
%   f\|_{L^2(\R^2)}^2\right)^{\frac12}
% \]
% is finite.

We now introduce $U^p, V^p$-type spaces that are adapted to the
linear Schr\"odinger flow.
\begin{defin}\label{def:delta_norm}
  For $s \in \R$, let $U^p_\Delta H^s $ (resp.~$V^p_\Delta H^s$) be
  the space of all functions $u:\R\to H^s(\R^2)$ such that $t \mapsto
  e^{-it \Delta}u(t)$ is in $U^p(\R,H^s)$ (resp.~$V^p(\R,H^s)$), with
  respective norms
  \begin{equation}\label{eq:delta_norm}
    \| u\|_{U^p_\Delta H^s} = \| e^{-it \Delta} u\|_{U^p(\R,H^s)},
    \qquad
    \| u\|_{V^p_\Delta H^s} = \| e^{-it \Delta} u\|_{V^p(\R,H^s)}
  \end{equation}
\end{defin}
%Spaces of this type have been successfully used as replacements for
%$X^{s,b}$ spaces and are still effective at critical scaling
%(see for instance \cite{KoTa05, KoTa07, HHK09}).

\begin{rem}\label{rem:ext}
  The embeddings in Remark \ref{rmk:prop_v}  and Lemma \ref{lem:interpol}
  naturally extend to the spaces $U^p_{\Delta}H^s$ and
  $V^p_{\Delta}H^s$.
\end{rem}

To clarify the roles of the $U^2_{\Delta},V^2_{\Delta}$ spaces, we introduce the
$X^{0, b}$-type spaces defined via the norms
\[
\|u\|_{\dot{X}^{0,\frac{1}{2},1}} = \sum_{\vartheta} \left(
\int_{|\tau-\xi^2|=\vartheta}|\tilde{u}(\tau,\xi)|^2|\tau-\xi^2|\,d\xi
\,d\tau \right)^{\frac{1}{2}}
\]
and
\[
\|u\|_{\dot{X}^{s,\frac{1}{2},\infty}} = \sup_{\vartheta} \left(
\int_{|\tau-\xi^2|=\vartheta}|\tilde{u}(\tau,\xi)|^2|\tau-\xi^2|\,d\xi
\,d\tau \right)^{\frac{1}{2}}
\]
For these spaces we have the embeddings \cite{KoTa07}
\begin{equation}\label{uvbesov}
\dot{X}^{0,\frac{1}{2},1}\subset U^2_{\Delta} L^2 \subset
V^2_{\Delta} L^2 \subset \dot{X}^{0,\frac{1}{2},\infty}
\end{equation}
From these inclusions we can conclude that the $U^2_{\Delta}$ and
$V^2_{\Delta}$ norms are equivalent when restricted in modulation
to a single dyadic scale.

Another straightforward consequence of the definitions (see for instance \cite[Proposition
    2.19]{HHK09}) is that one can extend the local smoothing estimate and Strichartz
    estimates to general $U^p_{\Delta}$ functions:
\begin{align} \label{est:u2localsmoothing}
\|e^{it\Delta}P_{N,\mathbf{e}}f\|_{L^{\infty,
2}_{\mathbf{e}}}
&\lesssim N^{-\frac{1}{2}}\|f\|_{U^2_{\Delta}} \\
\label{est:UpStrichartz}
\|e^{it\Delta}f\|_{L^q_tL^r_x(\mathbb{R}\times\mathbb{R}^2)}
&\lesssim
\|f\|_{U^p_{\Delta}}
\end{align}
Here $(q,r)$ is any admissible pair of exponents and $p := \min{(q,r)}$.

Finally, in the case of free solutions for the Schr\"odinger equation we can
easily do orthogonal frequency decompositions. For the $U^2_{\Delta}$ and $V^2_{\Delta}$
functions we have the following partial substitute:

\begin{lem}
\label{U2:sqsum}
Let $1 = \sum_R P_R(D)$ be a locally finite partition of unity in frequency,
with uniformy bounded symbols. Then we have
the dual bounds
\[
\sum_R \| P_R u\|_{U^2_\Delta}^2 \lesssim \|u\|_{U^2_\Delta}^2
\]
respectively
\[
\| \sum_R P_R f_R\|_{V^2_\Delta}^2 \lesssim  \sum_R \|
f_R\|_{V^2_\Delta}^2
\]
\end{lem}
The proof is straightforward and is left for the reader.

\subsection{Lateral $U^p$ and $V^p$ spaces}
Unfortunately the above function spaces are insufficient for
closing the multilinear estimates in our problem. Instead
we also need to  define the lateral $U^p$ and $V^p$ spaces.

Given a unit vector $\e \in \Sp^1$, we consider orthonormal coordinates
$(\xi_\e,\xi'_\e)$ with $\xi_\e = \xi \cdot \e$. Then we define the
Fourier region
\[
A_\e = \{ (\tau,\xi) \in \R^2; \ \xi \cdot \e > \frac14 |\xi|,
|\tau-\xi^2| < \frac{1}{32} \xi^2 \}
\]
In this region we have
\begin{equation}\label{xiereg}
\tau - \xi'^2_\e \geq \frac{1}{64}(|\tau|+\xi'^2_\e)
\end{equation}
and therefore can factor the symbol of the Schr\"odinger operator:
\[
\xi^2 - \tau =
(\xi_\e + \sqrt{\tau - \xi'^2_\e})(\xi_\e - \sqrt{\tau - \xi'^2_\e}) \approx |\xi|(\xi_\e - \sqrt{\tau - \xi'^2_\e})
\]
Hence instead of considering the forward Schro\"dinger evolution
we can work with the lateral flow
\[
\partial_\e - i L_\e, \qquad L_\e = \sqrt{ -i \partial_t - \partial'^2_\e}
\]
for functions frequency localized in the region \eqref{xiereg}.
We denote the corresponding $U^p$ and $V^p$ function spaces by
$U^p_\e$ and $V^p_\e$, respectively.

Now we are ready to define the nonlinear component of our function
spaces, namely $U^{2,\sharp}$ and $V^{2,\sharp}$. For that we need some
multipliers, denoted by $P_\e$, adapted to the regions $A_\e$.
The space $U^{2,\sharp}$ is given by
\begin{equation}\label{u2sharp}
U^{2,\sharp} = U^2_\Delta + |D|^{-\frac12} \Sigma_{\e}  P_\e U^2_\e
\end{equation}
In other words it can be thought of as an atomic space where the atoms
are normalized $U^2_\Delta$ functions and normalized $U^2_\e$ functions.

The space $V^{2,\sharp}$ is given by the norm
\begin{equation}\label{v2sharp}
\|\phi\|_{V^{2,\sharp}} = \|\phi\|_{V^2_\Delta} +
\sup_\e \| |D|^{\frac12}  P_\e \phi\|_{V^2_\e}
\end{equation}
By $U^{2,\sharp}_k$, respectively $V^{2,\sharp}_k$, we denote the corresponding
spaces of functions which are localized at frequency $2^k$.

The main properties of these spaces are summarized in the following

\begin{prop}\label{p:u2sharp}
The spaces $U^{2,\sharp}$ and $V^{2,\sharp}$ defined above have the following
properties:

a) Inclusion:
\begin{equation}\label{u2sharp-emb}
U^{2,\sharp} \subset V^{2,\sharp}
\end{equation}

b) Duality:
\begin{equation}\label{u2sharp-dual}
[(i \partial_t - \Delta) U^{2,\sharp}]^* = V^{2,\sharp}
\end{equation}

c) Truncation. For any time interval $I$ we have
\begin{equation}\label{Iu2sharp-emb}
\chi_I : U^{2,\sharp} \to V^{2,\sharp}
\end{equation}
\end{prop}

We postpone the proof of this result for later in the section.  Part
(a) is a special case of (c) when $I = \R$. Part (b) is a direct
consequence of the duality result in Lemma~\ref{lem:dual}. Part (c) is
proved in Lemma~\ref{u2sharp-v2}.

\subsection{The $l^2_k$ spatial structure.}
We need one additional structural layer to overlay on top of the
$U^2$ and $V^2$ structure, which has to do with the fact that
we are seeking to solve the problem locally in time. Thus all the
estimates we will have to prove apply to functions which are localized in
time to a compact interval. Within such an interval, waves at frequency
$2^k$ travel a distance of $O(2^k)$, with rapidly decreasing tails farther
out. Thus if we partition the space into $2^k$ sized squares, the interaction
of separated squares is negligible. This leads us to introduce
a  local in time partition of unity
\[
1 = \sum_{m \in \Z^2}  \chi_k^m(x,t), \qquad
\chi_k^m(x,t) = \chi_0(t) \chi(2^{-k} x - m)
\]
and corresponding norms
\[
\|\phi\|_{ l^2_k U^{2,\sharp} }^2  = \sum _{m \in \Z^2}
\| \chi_k^m(x,t) \phi\|_{U^{2,\sharp}}
\]
We similarly define  the $l^2_k V^{2,\sharp}$ and $ l^2_kDU^{2,\sharp} $ norms.
To relate these norms with the  previous ones we have the following:

\begin{prop}\label{p:l2Xsharp}
a) For all $\phi \in U_k^{2, \sharp}$ localized at frequency $2^k$ we have
\begin{equation}\label{xksharp-l2xk}
\|\phi\|_{l^2_k U^{2,\sharp}} \lesssim \|\phi\|_{U^{2,\sharp}}
\end{equation}

b) For all $\phi$ localized at frequency $2^k$ and with compact
support in time we have
\begin{equation}\label{xk-l2xk}
\|\phi\|_{ V^{2,\sharp}} \lesssim \|\phi\|_{l^2_k V^{2,\sharp}}
\end{equation}
\end{prop}

\subsection{The angular spaces}

The above spaces suffice in order to treat the nonlinear part of
$N(\phi,A)$. However, for the linear part $L$ we need an entirely different
type of structure. To set the notation, we denote the angular derivative
centered at $x_0$ by
\[
\snabla_{x_0} = (x-x_0) \wedge \partial_x
\]
We also set
\[
\la x-x_0 \ra_k = (2^{-2k}+(x-x_0)^2)^\frac12
\]
 Let $\sigma > 0 $ be a fixed constant. For $x_0 \in \R^2$   and
$k \in \Z$ we define the space $X^{x_0,\sigma}_{k}$ with norm
\begin{equation}\label{Xsigma0}
\| \phi\|_{X^{x_0,\sigma}_{k}} = 2^{k(\frac12-\sigma)} \|\la x-x_0
\ra_k^{-\frac12 - \sigma} \la \snabla \ra^{\sigma} \phi\|_{L^2}
\end{equation}
as well as the smaller space $X^{x_0,\sigma,\sharp}_{k}$ with norm
\begin{equation}\label{Xsigma0-sharp}
\| \phi\|_{X^{x_0,\sigma,\sharp}_{k}}= \|\phi(0)\|_{L^2} +
 \|(i \partial_t - \Delta) \phi\|_{X^{x_0,\sigma,*}_{k}}
\end{equation}

We further define
\begin{equation}\label{Xsigma}
X^{\sigma}_{k} = \bigcap_{x_0 \in \R^2} X^{x_0,\sigma}_{k},
\quad X^{\sigma,*}_{k} = \sum_{x_0 \in \R^2} X^{x_0,\sigma,*}_{k},
\quad X^{\sigma,\sharp}_{k} = \sum_{x_0 \in \R^2} X^{x_0,\sigma,\sharp}_{k},
\end{equation}
where the first space is the dual of the second.

These spaces are used for frequency $2^k$ solutions to the
Schr\"odinger equation. Their main properties are stated in the following

\begin{prop}\label{p:Xsigma}
The spaces defined above, restricted to frequency $2^k$ functions,
have the following properties:

a) Solvability:
\begin{equation}\label{Xsigma-solve}
\| \phi\|_{X^{\sigma,\sharp}_{k}} \lesssim
\|\phi(0)\|_{L^2} +  \|(i \partial_t - \Delta) \phi\|_{X^{\sigma,*}_{k}}
\end{equation}
b) Moving centers:
\begin{equation}\label{Xsigma-emb}
\|  \phi\|_{X^{\sigma}_{k}}+ \| \phi\|_{L^\infty L^2}  \lesssim \| \phi\|_{X^{\sigma,\sharp}_{k}}
\end{equation}
c) Nesting:
\begin{equation}\label{Xsigma-nest}
\|  \phi\|_{X^{\sigma_1}_{k}}  \lesssim \| \phi\|_{X^{\sigma_2}_{k}}, \qquad \sigma_2 < \sigma_1
\end{equation}
\end{prop}

In this result $k$ is a scaling parameter and can be set to $0$.
Part (a) is a straightforward consequence of the definitions.
However, part (b) is far less trivial, and requires two
separate estimates. First, for fixed $x_0$ we need to show that
\begin{equation}\label{Xsigma-solve-0}
\| \phi\|_{X^{x_0,\sigma}_{k}}+ \| \phi\|_{L^\infty L^2}  \lesssim \| \phi\|_{X^{x_0,\sigma,\sharp}_{k}}
\end{equation}
which is done in Lemma~\ref{Xsigma0-solve}.

Secondly, for $x_1 \neq x_0$ we need to show that
\begin{equation}\label{Xsigma-emb-0}
\| \phi\|_{X^{x_1,\sigma}_{k}} \lesssim \| \phi\|_{X^{x_0,\sigma}_{k}}
+ \| \phi\|_{X^{x_0,\sigma,\sharp}_{k}}
\end{equation}
This is achieved in Lemma~\ref{Xsigma0-shift}.
Part (c) follows from the similar property for fixed $x_0$, which is
straightforward in view of the frequency localization.

\subsection{Dyadic norms and the main function spaces.}

 Finally, we are ready to set up the global function spaces
where we solve the Chern-Simons-Schr\"odinger problem.
For the solutions at frequency $2^k$ we use two spaces.
The stronger norm $X_k^\sharp$ represents the space where
the solutions actually lie and is given by
\begin{equation}\label{Xksharp}
 X_k^\sharp = l^2_k U^{2,\sharp}_k + X^{\sigma,\sharp}_{k}
\end{equation}
Here $0<\sigma<\frac{1}{2}$ is a fixed constant.

However, this is a sum type space and so multilinear estimates would
be quite cumbersome, with many cases. Furthermore, the above space
is not stable with respect to time truncations. Instead we also
introduce a weaker topology
\begin{equation}\label{Xk}
 X_k = l^2_k V^{2,\sharp}_k \cap X^{\sigma}_{k}
\end{equation}
For the inhomogeneous term in the equation we have
the space $Y_k$ which has $X_k$ as its dual,
\begin{equation}\label{Yk}
Y_k = l^2_k DU^{2,\sharp}_k + X^{\sigma,*}_{k}
\end{equation}

The main result concerning our function spaces is in the following

\begin{theorem}\label{t:xkyk}
The following properties are valid for frequency $2^k$ functions:

a) Linear estimate:
\begin{equation}\label{Xk-solve}
\| \phi\|_{ X_k^\sharp} \lesssim \|\phi(0)\|_{L^2} + \| (i \partial_t - \Delta) \phi\|_{Y_k}
\end{equation}

b) Weaker norm:
\begin{equation}\label{Xk-emb}
\| \phi\|_{ X_k} \lesssim \| \phi\|_{ X_k^\sharp}
\end{equation}

c) Time truncation:
\begin{equation}\label{IXk}
\| \chi_I \phi\|_{ X_k} \lesssim \| \phi\|_{ X_k^\sharp}, \qquad I \subset \R
\end{equation}

d) Duality:
\begin{equation}\label{Xk-dual}
Y_k^* = X_k
\end{equation}
\end{theorem}

Part (a) is a direct consequence of the preceding three propositions.
The estimate in (b) is a special case of (c). Part (c)
also follows in part from the two preceding propositions.
However, we still need to address the cross embeddings,
\[
\chi_I l_k^2 U^{2,\sharp}_k \subset  X^{\sigma}_{k}
\]
respectively
\[
 \chi_I X^{\sigma,\sharp}_{k} \subset l^2_kV^{2,\sharp}_k
\]
The truncation in the first embedding can be harmlessly dropped as it
is bounded on $X^{\sigma}_{k}$. It remains to show  the embedding
\begin{equation}\label{u2toXs}
l^2_k U^{2,\sharp}_k
\subset X^{\sigma}_{k}
\end{equation}
which is proved in Lemma~\ref{U2-Xsigma}. The truncation in the second
embedding can also be dropped. To see this recall that $X^{\sigma,\sharp}_{k}
\subset L^\infty L^2$. This allows us to freely replace arbitrary functions
$u \in  X^{\sigma,\sharp}_{k} $
 by solutions to the homogeneous equation outside $I$. But then
$\chi_I u - u \in U^2_\Delta$ and we can use \eqref{u2sharp-emb}.  Once $\chi_I$ is dropped,
using again $X^{\sigma,\sharp}_{k}
\subset L^\infty L^2$, the problem reduces to
\[
X^{\sigma,*}_{k} \subset l^2_k DV^{2,\sharp}_k
\]
which follows from \eqref{u2toXs} by duality.

In this article we work with $H^s$ initial data. Correspondingly, we define
the spaces $X^{s,\sharp}$ and $X^s$ for solutions, respectively $Y^s$ for the nonlinearity,
by
\begin{equation}
 \| \phi\|_{ X^{s,\sharp}}^2 = \sum_{k \geq 0} 2^{2sk} \| P_k \phi\|_{ X^{\sharp}_k}^2
\end{equation}
\begin{equation}
  \| \phi\|_{ X^{s}}^2 = \sum_{k \geq 0} 2^{2sk} \| P_k \phi\|_{  X_k}^2
\end{equation}
\begin{equation}
 \| f \|_{ Y^{s}}^2 = \sum_{k\geq 0} 2^{2sk} \| P_k f\|_{  Y_k}^2
\end{equation}
where $P_0$ includes all frequencies less than $1$.

\subsection{The nonlinearity $N(\phi,A)$.}
Here we turn our attention to the nonlinear equation
\[
(i \partial_t - \Delta)\phi = N(\phi,A), \qquad \phi(0) = \phi_0,
\qquad A=A(\phi),
\]
where $A(\phi)$ is obtained by solving \eqref{At}-\eqref{A1A2Initial}.
We seek to solve this equation for positive $t$ and locally in time;
therefore we can harmlessly insert a cutoff function $\chi =
\chi_{[0,1]}$ in time and solve instead the modified equation
\begin{equation}\label{CSS-cut}
  (i \partial_t - \Delta)\phi = N(\chi \phi,A), \qquad \phi(0) = \phi_0,
\qquad A=A(\chi \phi)
\end{equation}
Any global solution to this modified equation will solve the original
equation in the time interval $[0,1]$.

We use the $H^s$ version of the linear estimate \eqref{Xk-solve} to
solve this equation in the space $X^{s,\sharp}$ using the contraction
principle. Thus we need to show that we have a small Lipschitz
constant for the map
\[
X^{s,\sharp} \ni \phi \to N(\chi \phi,A) \in Y^s, \qquad A=A(\chi
\phi)
\]
We subdivide this problem into two completely different problems,
which correspond to the decomposition of $N(\chi \phi,A) $ into a
linear and a nonlinear part.  To estimate the linear part $L (\chi
\phi)$ we will use the $H^s$ version of the embedding \eqref{IXk} and
select only the $X^{\sigma,s}$ part of the $X^s$ norm, neglecting the
$l^2$ structure. Then we can drop the cutoff $\chi$, and it remains to
prove the bound
\begin{equation}\label{lin-est}
  \| L \phi\|_{X^{\sigma,s,*}} \lesssim \|\phi\|_{ X^{\sigma,s}} \| \phi(0)\|_{H^s}^2
\end{equation}
This is achieved in Section~\ref{s:lin}, Proposition \ref{Cphi}.

To estimate the nonlinear part $Nl( \chi \phi,A) = N(\chi \phi,A) - L
\chi \psi$ we seek to prove %that the map
\[
X^{s,\sharp} \ni \phi \to Nl(\chi \phi,A) \in Y^s, \qquad A=A(\chi
\phi)
\]
Retaining only the $U^2$- $V^2$ part of our function spaces, it
suffices consider the map
\[
l^2 V^{2,\sharp,s} \ni \phi \to Nl(\phi,A) \in l^2 DU^{2,\sharp,s},
\qquad A=A( \phi)
\]
for $\phi$ localized in time. By duality, this translates to
Lipschitz continuity of the form
\begin{equation}\label{Nphipsi}
  l^2 V^{2,\sharp,s} \times l^2 V^{2,\sharp,{-s}} \ni (\phi,\psi) \to
  \int  Nl(\phi,A) \bar \psi \ dxdt
\end{equation}
To prove this we succesively consider all the terms in $Nl( \phi,A) $
in Sections~\ref{s:cubic}-\ref{s:last}.

\subsection{Linear estimates}
We now proceed to state and prove a collection of linear lemmas which,
together, imply the results stated before in this section.

Given an angle $A$ in $\R^2$ with opening less than $\pi$, we say that
a direction $\e$ is admissible with respect to $A$ if $\pm \e^\perp
\not\in A$.

For $ k \geq 0$ we define the following subset of Fourier space
\[
A_k = \{ (\xi,\tau) \in A \times \R;\ |\xi| \sim 2^k, \ |\tau -
\xi^2| \ll 2^{2k}\}
\]
We denote by $P_{A,k}$ a smooth space-time multiplier with support in $A_k$.
Then we have

\begin{lem}\label{u2sharp-v2}
  Let $A$ be an angle in $\R^2$, $k > 0$, $I$ a time interval and
  $\e_1$, $\e_2$ admissible directions with respect to $A$.
 Then for functions $f$ that are frequency localized in $A_k$, we have
  \begin{equation}
    \|P_{A,k}  \chi_I f\|_{V^2_{ \e_2}} \lesssim \| f\|_{U^2_{ \e_1}}
  \end{equation}
  with an implicit constant that is uniform with respect to pairs
  $\e_1,\e_2$ for which the distances $\text{dist}(\pm \e_{1,2}, A)$
  lie in a compact set away from zero.
\end{lem}

This lemma serves to prove the properties \eqref{u2sharp-emb} and
\eqref{Iu2sharp-emb} in Proposition~\ref{p:u2sharp}.  We note that two
nontrivial properties are coupled in the statement, namely the
embedding $U^2_{\e_1} \subset V^2_{ \e_2}$ and the time truncation.  We further
note that the same estimate holds true if either of the two lateral
spaces is replaced by the corresponding vertical space. In that case
the time truncation can be absorbed into the vertical space and one is
left with just the embedding, for which the proof below still applies.

\begin{proof}
  By scaling we can assume that $k=0$. We consider the multiplier
  $P$ which selects a small neighborhood of $A_0$. Then $P$ is bounded
  on both $V^2_{\e_2}$ and $U^2_{\e_1}$, and so it
  suffices to show that
  \begin{equation}
    \| P \chi_I P f\|_{V^2_{\e_2}} \lesssim \| f\|_{U^2_{\e_1}}
  \label{Pest}\end{equation}
Set $I = [t_0,t_1]$.  We can harmlessly replace $\chi_{I}$ by its
mollified version $Q_{\ll 0} \chi_I$ as its high modulation part
provides no output.

We first observe that the simpler bound
  \begin{equation}\label{u2-energy}
  \| P f\|_{L^{\infty,2}_{\e_2}}+ \| P f\|_{L^{\infty} L^2}
 \lesssim \| f\|_{U^2_{ \e_1}}
  \end{equation}
  follows easily by reducing to a $U^2_{ \e_1}$ atom
where the free waves associated to each step are supported
in a small neighborhood of the intersection of $A_0$ with the
paraboloid.  Then we
  apply  either the energy estimate or the lateral energy
  estimate in the $\e_2$ direction for each step of that atom.

  Then the bound \eqref{Pest} reduces to
  \[
  \| (i \partial_t - \Delta) P \chi_I P f\|_{DV^2_{ \e_2}} \lesssim \|
  f\|_{U^2_{ \e_1}}
  \]
  Using the duality between $DV^2$ and $U^2$, this is equivalent to
  the symmetric bound
  \begin{equation}
    |Q_R(f,g)| \lesssim \| f\|_{U^2_{ \e_1}}\| g\|_{U^2_{ \e_2}}
    \label{Pdual}\end{equation}
  where
\[
  Q_R(f,g) = \langle (i \partial_t - \Delta) \chi_I P f, P g \rangle =
\langle  P f, \chi_I  (i \partial_t - \Delta) P g \rangle
\]
This is not entirely symmetric, and so we also introduce its twin
\[
Q_L(f,g) = \langle \chi_I (i \partial_t - \Delta)  P f, P g \rangle
\]
Their difference is easy to control. Indeed, we have
\[
Q_R(f,g) - Q_L(f,g) = \langle [(i \partial_t - \Delta), \chi_I] P f, P g
\rangle =  \langle i\partial_t \chi_I P f, P g
\rangle
\]
The time derivative of $\chi_I$ is a sum of two unit bump functions on
a unit time interval around $t_0$, respectively $t_1$.  Hence using
the energy part of \eqref{u2-energy} we obtain
\begin{equation}\label{ql-r}
|Q_R(f,g) - Q_L(f,g)| \lesssim \| f\|_{U^2_{ \e_1}}\| g\|_{U^2_{ \e_2}}
\end{equation}
Given the support of $P$, we can rewrite $Q_L$ and $Q_R$  in terms of the
sideways evolutions for $f$ and $g$:
  \begin{equation}
\begin{split}
    Q_L(f,g) = & \  \langle \chi_I P  (D_{\e_1} - L_{\e_1})f, P g \rangle,
\\
   Q_R(f,g) = & \ \langle P f,  \chi_I P (D_{\e_1} - L_{\e_1}) g \rangle
\end{split}    \label{qlr}\end{equation}
Here the elliptic factor in the factorization of $i \partial_t - \Delta$  is included in $P$.
Thus by a slight abuse of notation we use the same $P$ for different
multipliers with similar size and support.

  It suffices to prove \eqref{Pdual} for atoms. Thus consider $f$ and $g$ of
  the form
  \[
  f = \sum \chi_{[a_i,b_i]} (x\cdot \e_1) f_i, \qquad g = \sum
  \eta_{[c_i,d_i]} (x\cdot \e_2) g_i
  \]
  where $f_i$ and $g_i$ are homogeneous waves, frequency localized in
  a small neighborhood of $A_1$, and with
  \[
  \sum_i \|f_i(0)\|_{L^2}^2 \approx 1, \qquad \sum_i \|g_i(0)\|_{L^2}^2 \approx 1
  \]
  As $f_i$ and $g_i$ are free waves frequency localized near the $A$
  section on the parabola at frequency one, we can measure their
  energy in an equivalent way at time $t=0$.

  Instead of the data at time $t=0$, it is better to describe $f_i$ in
  terms of its values at $x\cdot \e_1 = a_i$ and at $x\cdot \e_1 =
  b_i$.  By a slight abuse of notation we denote these two functions
  by $f_i(a_i)$ and $f_i(b_i)$.  We remark that $f_i(a_i)$ and
  $f_i(b_i)$ are related via the sideways evolution and in particular
  we have
  \[
  \|f_i(a_i)\|_{L^2} = \|f_i(b_i)\|_{L^2}
  \]
  However, it will be convenient to work with both of them together
  rather than separately.

  We observe that it suffices to consider the case when $b_i - a_i \gg
  1$ and $c_i - d_i \gg 1$. Indeed, if for instance $b_i - a_i
  \lesssim 1$ for all $i$ then
  \[
  \|f\|_{L^2} \lesssim 1
  \]
  This is easily combined with the following easy consequence of \eqref{uvbesov},
  \[
  \|(i \partial_t - \Delta)P g\|_{L^2} \lesssim 1,
  \]
  to conclude the argument.

  We can also assume without any restriction in generality that
  $a_{i+1}-b_i \gg 1$ and $c_{i+1}-d_i \gg 1$.  To the intervals
  $[a_i,b_i]$ we associate bump functions $\chi_i$ which equal $1$
  inside the interval and decay rapidly on the unit scale.  By
  $\eta_i$ we denote similar bump functions associated to $[c_i,d_i]$.
Set
\[
\begin{split}
 B^{ij}_{L}:= & \ Q_L( \chi_{[a_i,b_i]} (x\cdot \e_1) f_i, \eta_{[c_j,d_j]} (x\cdot \e_2) g_j) \\
B^{ij}_{R}:= & \ Q_R( \chi_{[a_i,b_i]} (x\cdot \e_1) f_i, \eta_{[c_j,d_j]} (x\cdot \e_2) g_j)
\end{split}
\]
We want to be able  to use $Q_L$ and $Q_R$ interchangeably. For that
we estimate the difference
\[
 B^{ij}_{L} - B^{ij}_{R} =  \langle i\partial_t \chi_I P\chi_{[a_i,b_i]} (x\cdot \e_1) f_i , P  \eta_{[c_j,d_j]} (x\cdot \e_2) g_j
\rangle
\]
Using the time localization given by $\partial_t \chi_I$ and the finite
speed of propagation in time for waves supported in $A_0$, we
obtain a localized analogue of \eqref{ql-r}, namely
\[
| B^{ij}_{L} - B^{ij}_{R}| \lesssim \| \chi_i \eta_j f_i(t_0)\|_{L^2}
\| \chi_i \eta_j g_j(t_0)\|_{L^2} +  \| \chi_i \eta_j f_i(t_1)\|_{L^2} \| \chi_i \eta_j
g_j(t_1)\|_{L^2}
\]
By Cauchy-Schwarz this implies that
\begin{equation}
  \label{bij:l-r}
\sum_{i,j} | B^{ij}_{L} - B^{ij}_{R}| \lesssim 1
\end{equation}
which indeed allows us to estimate $B^{ij}_{L}$ and $B^{ij}_{R}$
interchangeably. Using the representation of $Q_L$ in \eqref{qlr}
we have
  \[
  \begin{split}
    B_L^{ij}:= & \ Q_{L}( \chi_{[a_i,b_i]} (x\cdot \e_1) f_i,
\eta_{[c_j,d_j]} (x\cdot \e_2) g_j) \\
    = & \ \langle f_i(b_i) \delta_{x\cdot \e_1=b_i} - f_i(a_i)
    \delta_{x\cdot \e_1=a_i}, P(\eta_{[c_j,d_j]} (x\cdot \e_2) g_j)
    \rangle
  \end{split}
  \]
A symmetric formula holds for $B_R^{ij}$.
  For $g_j$ we have lateral energy estimates in the $\e_1$ directions,
  and $P$ has a rapidly decreasing kernel. Hence the above expression
  is bounded by
  \begin{equation}
    |B_{ij}^L| \lesssim \|  \eta_j f_i(b_i)\|_{L^2}
\|  \eta_j g_j(b_i)\|_{L^2}
 +
\|  \eta_j f_i(a_i)\|_{L^2}
\|  \eta_j g_j(a_i)\|_{L^2}
    \label{bleft}\end{equation}
 In order to complete the proof of \eqref{Pdual} for atoms
 we need to distinguish
between different interval balances:

{\bf A. Unbalanced intervals:} Either $b_i - a_i \gg d_j - c_j$ or
  $b_i - a_i \ll d_j - c_j$. In this case we will prove that
\begin{equation}\label{bwish}
    \begin{split}
      \min\{ |B_L^{ij}|,|B_R^{ij}|\}  \lesssim&\; (\| \tilde \eta_j(x \cdot \e_2)
      f_i(b_i)\|_{L^2} + \| \tilde \eta_j(x \cdot \e_2)
      f_i(a_i)\|_{L^2}) \\
      &
      (\| \tilde  \chi_i(x \cdot \e_1) g_j(d_j)\|_{L^2}+ \| \tilde \chi_i(x \cdot \e_1) g_j(c_j)\|_{L^2})
    \end{split}
  \end{equation}
  for some more relaxed bump functions $\tilde \eta_j$ and $\tilde
  \chi_i$ which share the properties of $\chi_i$ and $\eta_j$.  Assuming
  \eqref{bwish} is  true, the estimate for the corresponding part  of \eqref{Pdual} easily follows  
  from Cauchy-Schwarz:
  \[
  \begin{split}
    \sum_{i,j}  \min\{ |B_L^{ij}|,|B_R^{ij}|\}
\lesssim & \ \sum_{i,j} \| \tilde \eta_j(x
    \cdot \e_2) f_i(b_i)\|_{L^2}^2 + \| \tilde \eta_j(x \cdot \e_2)
    f_i(a_i)\|_{L^2}^2 \\ & \ +\sum_{i,j} \| \tilde \chi_i(x \cdot
    \e_1) g_j(d_j)\|_{L^2}^2 + \| \tilde \chi_i(x \cdot \e_1)
    g_j(c_j)\|_{L^2}^2 \\ \lesssim & \ \sum_{i} \| f_i(b_i)\|_{L^2}^2
    + \| f_i(a_i)\|_{L^2}^2 +\sum_{j} \| g_j(d_j)\|_{L^2}^2 + \|
    g_j(c_j)\|_{L^2}^2 \\ \lesssim & \ 1
  \end{split}
  \]

By symmetry suppose that  $b_i - a_i \gg d_j - c_j$. Then \eqref{bwish}
follows from \eqref{bleft} due to the propagation estimate
\[
\|  \eta_j g_j(b_i)\|_{L^2}+ \|  \eta_j g_j(a_i)\|_{L^2}
\lesssim \| \chi_i  g_j(d_j)\|_{L^2}+ \| \chi_i  g_j(c_j)\|_{L^2}
\]
To see this it suffices to consider the Schr\"odinger propagator from
the surfaces $x \cdot \e_1 = a_i, b_i$ to the surfaces $x \cdot \e_2 =
c_j, d_j$. corresponding to waves which are localized in $A_0$.  On
the one hand, with respect to suitable elliptic multiplier weights, this
is an $L^2$ isometry. On the other hand, its kernel decays rapidly
outside a conic neighborhood of the propagation cone associated to
$A_0$.  Hence all that remains to be seen is that the propagation cone
of the interval $\{ x \cdot \e_1 = a_i, \ x \cdot \e_2 \in
[c_j,d_j]\}$ either intersects the line $x \cdot \e_2 = c_j$ within
the interval $ x \cdot \e_1 \in [a_i,b_i]$ or intersects the line
$x \cdot \e_2 = d_j$ within the interval $ x \cdot \e_1 \in
[a_i,b_i]$. But this is a geometric consequence of the unbalanced
intervals.

  {\bf B. Balanced intervals.}  Here we consider the case when $b_i -
  a_i \sim d_j - c_j$.  The first observation is that it suffices to
  consider a fixed dyadic scale $X$ and assume that
  \[
  b_i - a_i \sim d_j - c_j \sim X
  \]
  The dyadic summation with respect to $X$ will be straightforward
  since we have $l^2$ summbability both on the $f$ and on the $g$
  side.

  The simplification that occurs when we fix the interval size is that
  we are allowed to relax the localization scale in the choice of the
  functions $\chi_j$ and $\eta_j$ in \eqref{bleft}.  Precisely,
  instead of the rapid decay on the unit scale (dictated by the
  smallest distance to the next interval) we allow them to decay
  rapidly on the $X$ scale, and denote them by $\chi_i^X$ and
  $\eta_j^X$. This makes the following norms
  equivalent:
  \[
  \| \eta_j^X f_j(b_i)\|_{L^2}
\approx \| \eta_j^X f_j(a_i)\|_{L^2}
\approx  \|\chi_i^X f_j(c_j)\|_{L^2} \approx
 \|\chi_i^X f_j(c_j)\|_{L^2}
 \]
  by standard propagation arguments.

  By \eqref{bleft} this implies the version of \eqref{bwish} with the
  weights $\chi_i^X$ and $\eta_j^X$.  The punch line is then in the
  $i$ and $j$ summation argument under \eqref{bwish}.  The bumps $
  \chi_i^X$ and $ \eta_j^X$ are wider now, but they are still almost
  orthogonal since the intervals are now also uniformly spaced at
  distance $X$ (or above).

\end{proof}

Our next lemma serves to prove the estimate \eqref{Xsigma-solve-0}, which
is needed for Proposition~\ref{p:Xsigma}. In order to do that we need
two more definitions, namely the local energy space (centered at $0$)
$LE$ and its dual $LE^*$. The $LE$ space-time norm adapted to
frequency-one functions is defined as
\[
\|\phi \|_{LE} = \|\phi\|_{L^2(|x| \lesssim 1)} +
\sup_{j >0} 2^{-\frac{j}2}  \|\phi\|_{L^2(|x| \approx 2^j)}
\]
Then we have

\begin{lem}\label{Xsigma0-solve}
  Let $s > 0$. Then for frequency-one functions $u$ solving $(i \partial_t - \Delta) u =
  f_1+f_2$, where $f_1$ has no radial modes,  we have the following estimate
  \begin{equation}\label{ang-solve}
    \| u \|_{LE} + \| \langle r\rangle^{-\frac12 - s } \langle \snabla\rangle^{s} u
    \|_{L^2} \lesssim
    \| u(0)\|_{L^2} + \|  \langle r\rangle^{\frac12 + s } \langle \snabla\rangle^{-s} f_1 \|_{L^2} + \|f_2\|_{LE^*}
  \end{equation}
\end{lem}

\begin{proof}
  Our starting point is the standard local energy decay estimate for
  frequency-one functions, namely
  \begin{equation}\label{le-solve}
    \| u \|_{LE}   \lesssim
    \| u(0)\|_{L^2} + \|f \|_{LE^*}
  \end{equation}
  We expand the function $u$ in \eqref{ang-solve} in an angular
  Fourier series. This preserves the frequency localization, and it
  suffices to prove \eqref{ang-solve} for each such mode separately
  (with uniform constants).

  Our contention is that for a fixed angular mode the bound
  \eqref{ang-solve} is a direct consequence of \eqref{le-solve}.  To
  see that let $k \in \Z $ and $u$ be of the form
  \[
  u_k(t,x) = u_{k}(r,t) e^{ i k \theta}
  \]
  Then we have
  \[
  \la r\ra ^{-\frac12 - s } \snabla^{s} u_k = \la r\ra^{-\frac12 - s } k^{s} u_k
  \]
  This is easily controlled by the $LE$ norm of $u$ for $r \gtrsim
  k$. To deal with smaller $r$ we need to use the angular
  localization.  Precisely, we claim that
  \begin{equation} \label{rad-proj} \| r^{-\frac12 - s } P_0
    u_k\|_{L^2} \lesssim \| (r+k)^{-\frac12 - s } u_k\|_{L^2}
  \end{equation}
  which easily leads to
  \[
  k^s \| \la r \ra^{-\frac12 - s } P_0 u_k\|_{L^2} \lesssim \|u_k\|_{LE}
  \]
  and, by duality,
  \[
  \|P_0 f_k\|_{LE^*} \lesssim k^{-s} \| \la r\ra ^{\frac12 + s } f_k\|_{L^2}
  \]
 The last two bounds prove that \eqref{ang-solve} follows from
 \eqref{le-solve}.

  It remains to establish \eqref{rad-proj}. The kernel of $P_0$ is
  given by a Schwartz function $\phi$. Then for $|x| \lesssim k$ we
  write
  \[
  P_0 u_k (x) = \int \phi(x-y) u(y) dy = (-k)^{-N} \int \snabla_y^N
  \phi(x-y) u(y) dy.
  \]
  We have
  \[
  |\snabla_y^N \phi(x-y)| \lesssim \langle x \rangle^{N} \langle x-y
  \rangle^{-N}
  \]
  and therefore
  \[
  |P_0 u_k (x)| \lesssim \left(\frac{\langle r \rangle}{|k|}\right)^N
  \int \langle x-y \rangle^{-N} | u(y)| dy
  \]
  Hence \eqref{rad-proj} easily follows.

\end{proof}

As a consequence of the above lemma we get the following result, which proves
the embedding \eqref{u2toXs} needed in Theorem~\ref{t:xkyk}.

\begin{lem}\label{U2-Xsigma}
  The following inequality holds for frequency $2^k$ functions $u$:
  \begin{equation}\label{eq:U2-Xsigma}
    2^{(\frac12-s)k}
\|\langle r\rangle_k^{-\frac12 - s } \langle\snabla\rangle^{s} u \|_{L^2}
    \lesssim \|u\|_{U^{2,\sharp}}
  \end{equation}
In addition, if $u$ is localized in a unit time interval then
 \begin{equation}\label{l2U2-Xsigma}
   2^{(\frac12-s)k}  \|\langle r\rangle_k^{-\frac12 - s }
\langle\snabla\rangle^{s} u \|_{L^2}
   \lesssim \|u\|_{l^2_k U^{2,\sharp}}
  \end{equation}

\end{lem}

\begin{proof}
  It suffices to consider $0 < s < \frac12$.  For the vertical $U^2_\Delta$
  space the bound \eqref{eq:U2-Xsigma} is a direct consequence of the
  previous lemma via the atomic decomposition. It remains to consider
  a lateral $U^2_\e$ space and a corresponding atom
  \[
  u = \sum \chi_j u_j
  \]
  For each of these atoms we have
  \[
  2^{(\frac12-s)k} \|\langle r\rangle_k^{-\frac12 - s } \langle
  \snabla \rangle^{s} u_j \|_{L^2} \lesssim \|u_j(0)\|_{L^2}
  \]
  Thus it would suffice to show that
  \[
  \|\langle r\rangle_k^{-\frac12 - s } \langle \snabla \rangle^{s} (\sum
  \chi_j u_j) \|_{L^2}^2 \lesssim \sum \|\langle r\rangle^{-\frac12 - s }
  \langle \snabla \rangle^{s}u_j \|_{L^2}^2
  \]
The bound \eqref{l2U2-Xsigma} also reduces to the same estimate,
with the only difference being that the atomic decomposition
is now done separately in each $2^k$ sized spatial cube.

The last bound reduces to an estimate on the unit circle,
  \[
  \|\sum \chi_j u_j\|_{H^s(\mathbb S^1)}^2 \lesssim \sum \|
  u_j\|_{H^s(\mathbb S^1)}^2, \qquad 0 \leq s < \frac12
  \]
  It makes no difference whether this is done on the circle or on the
  real line. The following argument is for the case of the real line.
  We begin with a simple observation, namely that
  \[
  \|\chi_j u_j\|_{H^s} \lesssim \|u_j\|_{H^s}
  \]
  Hence we can drop the cutoffs $\chi_j$ and instead assume that the
  $u_j$ have disjoint supports in consecutive intervals $I_j$. We use
  a Littlewood-Paley decomposition, but instead of having sharp
  Fourier localization it is convenient to choose multipliers $P_k$
  whose kernels have sharp localization in the physical space on the
  $2^{-k}$ scale.

  To estimate $P_k(\sum u_j)$, we split the intervals $I_j$ into long
  ($2^k |I_j| > 1$) and short ($2^k |I_j| < 1$). The outputs of long
  intervals are almost orthogonal,
  \[
  \| P_k (\sum_{I_j \ long} u_j) \|_{L^2}^2 \lesssim \sum_{I_j \ long}
  \|P_k u_j\|_{L^2}^2
  \]
  It remains to consider the outputs of short intervals. We still have
  orthogonality at interval separations of $2^{-j}$, and so we can write
  \[
  2^{2sk} \| P_k (\sum_{I_j \ short} u_j) \|_{L^2}^2 \lesssim 2^{2sk}
  \sum_{|I| = 2^{-k}} \left(\sum_{I_j \ short}^{I_j \subset 2I} \| P_k u_j
    \|_{L^2}\right)^2
  \]
  But for short intervals we can use the fact that the kernel of $P_k$
is a bump function with $2^{-k}$ sized support and $2^k$ amplitude to write
  \[
  \| P_k u_j \|_{L^2} \lesssim |I_j|^\frac12 2^{k/2} \|u_j\|_{L^2}
  \lesssim |I_j|^{\frac12+s} 2^{k/2} \|u_j\|_{H^s}
  \]
  Hence we obtain
  \[
  \begin{split}
    2^{2sk} \| P_k (\sum_{I_j \ short} u_j) \|_{L^2}^2 \lesssim & \
    \sum_{|I| = 2^{-k}} \left(\sum_{I_j \subset 2I} (2^k
      |I_j|)^{\frac12+s} \|u_j\|_{H^s}\right)^2
    \\
    \lesssim & \ \sum_{|I| = 2^{-k}} \left(\sum_{I_j \subset 2I} 2^k
      |I_j|\right ) \left(\sum_{I_j \subset 2I} (2^k |I_j|)^{2s}
      \|u_j\|_{H^s}^2\right)
    \\
    \lesssim & \sum_{2^k |I_j| \leq 1} (2^k |I_j|)^{2s}
    \|u_j\|_{H^s}^2
  \end{split}
  \]
  and the $k$ summation is straightforward.

\end{proof}
Finally, the following lemma allows us to move centers and proves
the estimate \eqref{Xsigma-emb-0}, which is needed in Proposition~\ref{p:Xsigma}.:

\begin{lem}\label{Xsigma0-shift}
  For $u$ localized at frequency one solving $(i \partial_t - \Delta) u =
  f$  and for any $x_0 \in \R^2$, we have
  \begin{equation}
    \| \langle x-x_0\rangle^{-\frac12 - s } \langle \snabla_{x_0}\rangle^{s} u
    \|_{L^2} \lesssim
    \| u(0)\|_{L^2} + \|  \langle r\rangle^{\frac12 + s }
\langle \snabla\rangle^{-s} f \|_{L^2}
  \end{equation}

\end{lem}

\begin{proof}
  We split $u$ into several spatial regions depending on the ratio of
  $\langle x-x_0\rangle$ and $R =\langle x_0\rangle$.

  (i) The intermediate region, $ A_{med} = \{\langle x-x_0\rangle
  \approx R\}$.  In this region it suffices to use the local energy
  decay
  \[
  \| \langle x-x_0\rangle^{-\frac12 - s } \langle
  \snabla_{x_0}\rangle^{s} \chi_{med}u\|_{L^2} \lesssim R^{-\frac12}
  \| \chi_{med}u\|_{L^2} \lesssim \|u\|_{LE}
  \]
  and then the previous lemma.

  (ii) The inner region, $ A_{in} = \{\langle x-x_0\rangle \ll R\}$.
  Here we compute
  \[
  (i \partial_t - \Delta )(\chi_{in} u) = f_{in}:= -2 \nabla
  \chi_{in}\cdot \nabla u - \Delta \chi_{in} \cdot u + \chi_{in} f
  \]
  and use the local energy bound for $u$ to estimate
  \[
  \| f_{in}\|_{LE^*} \lesssim \|u\|_{LE} + \| \chi_{in} f \|_{LE^*}
  \lesssim \|u\|_{LE} + \| \langle r\rangle^{\frac12 + s } \langle
  \snabla\rangle^{-s} f \|_{L^2},
  \]
  where at the last step we have used the fact thet $\chi_{in}$ is
  supported in a single dyadic region $\langle x \rangle \approx R$.
Then we apply Lemma~\ref{Xsigma0-solve}.

  (iii) The outer region, $ A_{out} = \{\langle x-x_0\rangle \gg R\}$.
  Here we use the following estimate which applies for frequency one
functions
  \[
  \| \langle x-x_0\rangle^{-\frac12 - s } \langle
  \snabla_{x_0}\rangle^{s} \chi_{out} u \|_{L^2} \lesssim \| \langle
  x\rangle^{-\frac12 - s } \langle \snabla \rangle^{s} u \|_{L^2} + R^s \|
  \langle x+R\rangle^{-\frac12 - s }
  u\|_{L^2},
  \]
and estimate the second term on the right by the local energy norm.
This in turn  is proved by complex interpolation
 between $s = 0$ and $s = 1$ since
  \[
  \snabla - \snabla_{x_0} = x_0 \cdot \nabla
  \]
\end{proof}

%%%%%%%%%%%%%%%%%%%%%%%%%%%%%%%%%%%%%%%%%%%%%%%%%%%%%%%
%%%%%%%%%%                                                    
%%%%%%%%%%                      The linear part                     
%%%%%%%%%%                                                  
%%%%%%%%%%%%%%%%%%%%%%%%%%%%%%%%%%%%%%%%%%%%%%%%%%%%%%%
\section{The linear part of $N(\phi,A)$}  \label{s:lin}
In this section we prove the main estimate \eqref{lin-est} for the
component $L$ of the nonlinearity, see \eqref{L}. For convenience
we restate the full result here:

\begin{prop}
Let $C = H^{-1} \Delta^{-1} |\phi_0|^2$. Then for $s \geq 0$ we have
\begin{equation} \label{Cphi}
\|Q_{12}(C,\phi)\|_{X^{\sigma,s,*}} \lesssim \|\phi\|_{ X^{\sigma,s}}
\|\phi_0\|_{H^s}^2
\end{equation}
\end{prop}

\begin{proof}
For $u_0 =  |\phi_0|^2$ we use the multiplicative Sobolev estimate
\[
\| u_0\|_{L^1} + \|u_0\|_{B^{1,\infty}_s} \lesssim \|\phi_0\|_{H^s}^2
\]
This is somewhat wasteful if $s > 0$ but it is tight for $s = 0$.
Using duality we rewrite the bound \eqref{Cphi} in the more symmetric form
\[
\left|\int Q_{12}(C,\phi) \psi dx dt\right| \lesssim
\|\phi\|_{ X^{\sigma,s}} \|\psi\|_{ X^{\sigma,-s}}
\|\phi_0\|_{H^s}^2
\]
Integrating by parts it is easy to see that the null form $Q_{12}$ can
be placed on any two of the factors $C,\phi,\psi$.  We use the
standard Littlewood-Paley trichotomy.

\bigskip

{\bf A. High-low interactions.}
Here we only need Bernstein's inequality to write for $k > j$
\[
\begin{split}
\left|\int Q_{12}(C_k,\phi_j) \psi_k dx dt\right|
\lesssim & \ 2^{k+j} \|C_k\|_{L^1} \|\phi_j\|_{L^\infty}
\|\psi_k\|_{L^\infty}
\\ \lesssim & \  2^{2(j-k)} 2^{-sk} \|u_0\|_{B^{1,\infty}_s}
 \|\phi_j\|_{L^\infty L^2}\|\psi_k\|_{L^\infty L^2}
\\ \lesssim & \  2^{2(j-k)} 2^{-sj}  \|\phi_0\|_{H^s}^2
 \|\phi_j\|_{ X^{\sigma,s}} \|\psi_k\|_{{ X^{\sigma,-s}} }
\end{split}
\]
The factor $2^{-sj}$ is not needed. The summation with respect to $k$
and $j$ is ensured by the off-diagonal decay.

\bigskip

{\bf B. High-high interactions.}
This case is equivalent to the one above if $s = 0$ and better if $s > 0$.

\bigskip

{\bf C. Low-high interactions.}
In this case it suffices to prove the estimate
\begin{equation}\label{q12pp}
\left|\int Q_{12}(C_{<k},\phi_k) \psi_k dx dt\right| \lesssim
\| \phi_k\|_{X^\sigma_k} \| \psi_k\|_{X^\sigma_k} \|u_0\|_{L^1}
\end{equation}
for some choice of $\sigma$. This choice is not important due to the
nesting property of the $X^\sigma_k$ spaces. It is easiest to work
with $\sigma =\frac12$.

By scaling we can take $k=0$. By translation invariance
we can take $u_0 = \delta_0$. Then $C$ is radial, and
\[
\nabla C_{< 0}(x) = x a(|x|,t), \qquad |a(r,t)| \lesssim
(1 + r+ t^\frac12)^{-2}
\]
Hence
\[
 Q_{12}(C_{<0},\phi_0) = a(|x|,t) \snabla \phi_0
\]
which shows that
\[
\| Q_{12}(C_{<0},\phi_0)\|_{X^{\frac12,*}_{0}} \lesssim
\| \phi_0\|_{X^{\frac12}_{0}}
\]
Thus \eqref{q12pp} follows.

\end{proof}

%%%%%%%%%%%%%%%%%%%%%%%%%%%%%%%%%%%%%%%%%%%%%%%%%%%%%%%
%%%%%%%%%%                                                    
%%%%%%%%%%                      Bilinear estimates                  
%%%%%%%%%%                                                   
%%%%%%%%%%%%%%%%%%%%%%%%%%%%%%%%%%%%%%%%%%%%%%%%%%%%%%%
\section{Bilinear estimates}  \label{S:BilinearEstimates}
We define the  temporal frequency localization operator $\qo_N$ to be the 
Fourier multiplier with symbol $\psi_N(\tau)$ and the modulation localization operator $Q_N$ to be 
the Fourier multiplier with symbol $\psi_N(\tau-\xi^2)$. 
Here $\psi_N$ is the same bump function we used in (\ref{PNsymbol}) to define Littlewood-Paley projections.

In the rest of the paper, $Q_N$ will be applied to single functions whereas $\qo_N$
will be applied to bilinear expressions.  
 
In the following, we will sometimes use the notation $U^2$, which stands for both $U^2_\Delta$ and $\due$. 
The same convention holds for  $V^2$.

 \subsection{Pointwise bilinear estimates}

These are needed for the case of balanced frequency interactions.
Let $I_\lambda$ denote the frequency annulus
$\{ \xi \in \R^2 : \lambda / 2 \leq \lvert \xi \rvert \leq 2\lambda \}$.
Our first result is
\begin{lem}\label{inftybound}
Let $\mu, \nu, \lambda$ be dyadic frequencies satisfying $\mu
\ll \lambda$ and $\nu \lesssim \mu \lambda$. Let
$\phi_\lambda, \psi_\lambda$ be functions with frequency support
contained in $I_\lambda$. Then
\begin{equation}\label{eq:inftybound}
\lVert P_{\mu} \qo_{\nu} (\bar{\phi}_\lambda \psi_\lambda) \rVert_{L^\infty}
\lesssim
\frac{\mu \nu}{\lambda}
\lVert \phi_\lambda \rVert_{V^{2,\#}}
\lVert \psi_\lambda \rVert_{V^{2,\#}}
\end{equation}
\end{lem}
\begin{proof}
We first dispense with the high modulations in the inputs.
If both are high ($\gtrsim \mu \lambda$) then by Bernstein we have
\[
\lVert P_{\mu} \qo_{\nu} (\overline{Q_{\gtrsim \mu \lambda}
  \phi}_\lambda Q_{\gtrsim \mu \lambda}\psi_\lambda) \rVert_{L^\infty}
\lesssim \nu \mu^2 \| (\overline{Q_{\gtrsim \mu \lambda} \phi}_\lambda
Q_{\gtrsim \mu \lambda}\psi_\lambda) \rVert_{L^1}
\lesssim \frac{\mu \nu}{\lambda}
\lVert \phi_\lambda \rVert_{V^{2}_\Delta}
\lVert \psi_\lambda \rVert_{V^{2}_\Delta}
\]
If one is high and one is low then we decompose the low modulation
factor with respect to small angles and use the lateral energy:
\[
\lVert P_{\mu} \qo_{\nu} (\overline{P_\e
  \phi}_\lambda Q_{\gtrsim \mu \lambda}\psi_\lambda) \rVert_{L^\infty}
\lesssim \nu \mu^{\frac32} \| (\overline{P_\e \phi}_\lambda
Q_{\gtrsim \mu \lambda}\psi_\lambda) \rVert_{L^{2,1}_\e}
\lesssim \frac{\mu \nu}{\lambda}
\lVert \phi_\lambda \rVert_{\dve}
\lVert \psi_\lambda \rVert_{V^{2}_\Delta}
\]
Finally if both factors are low modulations then we decompose
with respect to small angles and compute
\[
\lVert P_{\mu} \qo_{\nu} (\overline{P_\e
  \phi}_\lambda P_\e \psi_\lambda) \rVert_{L^\infty}
\lesssim \nu \mu \| (\overline{P_\e \phi}_\lambda
P_\e \psi_\lambda) \rVert_{L^{\infty,1}_\e}
\lesssim \frac{\mu \nu}{\lambda}
\lVert P_\e \phi_\lambda \rVert_{\dve}
\lVert P_\e \psi_\lambda \rVert_{\dve}
\]
\end{proof}

A slight sharpening of the above result is as follows:
\begin{lem}\label{Hinftybound}
Let $\mu, \lambda$ be dyadic frequencies satisfying $\mu
\ll \lambda$.  
%and $\nu \lesssim \mu \lambda$.
Let
$\phi_\lambda, \psi_\lambda$ be functions with frequency support
contained in $I_\lambda$. Then
\begin{equation}\label{eq:Hinftybound}
\lVert P_{\mu} H^{-1} (\bar{\phi}_\lambda \partial_x \psi_\lambda) \rVert_{L^\infty}
\lesssim
\mu \lVert \phi_\lambda \rVert_{V^{2,\#}} \lVert \psi_\lambda \rVert_{V^{2,\#}}
\end{equation}
\end{lem}
\begin{proof}
  While using Bernstein's inequality as in the previous proof leads to
  a logarithmic divergence and is no longer immediately useful, we can instead
use kernel bounds for $ P_{\mu} H^{-1} $ with the same effect.  The kernel
$K_\mu$ of $P_{\mu} H^{-1} $ satisfies
\[
|K_\mu(t,x)| \lesssim \mu^{2} (1+ \mu |x|)^{-N} (1+ \mu^2
|t|)^{-N}
\]
Then one can repeat the three cases in the previous proof, but using the kernel bounds
instead of Bernstein's inequality.
\end{proof}

\subsection{$L^2$ bilinear estimates for free solutions}

We introduce an improved bilinear Strichartz estimate
that is a slight generalization of that first shown in \cite[Lemma 111]{Bo98}.
\begin{lem}[Improved bilinear Strichartz]\label{L:BilinearEstimate}
Let $u(x, t) = e^{it \Delta}u_0(x), v(x, t) = e^{i t \Delta}v_0(x)$,
where $u_0, v_0 \in L^2(\R^2)$.
Let $\Omega_1$ denote the support
of $\hat{u}_0(\xi_1)$, $\Omega_2$ the support of $\hat{v}_0(\xi_2)$,
and set $\Omega = \Omega_1 \times \Omega_2$.
Assume that $\Omega_1$ and $\Omega_2$ are
open and
separated by some
positive distance. Then
\begin{equation}
\lVert u \bar{v} \rVert_{L^2_{t,x}}
\lesssim
\left( \frac{\sup_{\xi, \tau}
\int_{ \substack{\xi = \xi_1 - \xi_2 \\ \tau = \lvert \xi_1 \rvert^2 - \lvert \xi_2 \rvert^2}}
\chi_{\Omega}(\xi_1, \xi_2)
d\mathcal{H}^1(\xi_1, \xi_2)}
{\dist(\Omega_1, \Omega_2)} \right)^{1/2}
\cdot
\lVert u_0 \rVert_{L^2} \lVert v_0 \rVert_{L^2}
\end{equation}
where $d\mathcal{H}^1$ denotes 1-dimensional Hausdorff measure
(on $\R^4$)
and
$\chi_\Omega(\xi_1, \xi_2)$ the characteristic function of $\Omega$.
\end{lem}
\begin{proof}
To control $\lVert u \bar{v} \rVert_{L^2_{t,x}}$, we are led by duality to estimating
\begin{equation*}
\int_{\xi_1, \xi_2} g(\xi_1 - \xi_2, \lvert \xi_1 \rvert^2 - \lvert \xi_2 \rvert^2)
\hat{u}_0(\xi_1) \bar{\hat{v}}_0(\xi_2) d\xi_1 d\xi_2
\end{equation*}
We apply Cauchy-Schwarz and reduce
the problem to bounding
\begin{equation*}
G :=
\int_{(\xi_1, \xi_2) \in \Omega}
\lvert
g(\xi_1 - \xi_2, \lvert \xi_1 \rvert^2 - \lvert \xi_2 \rvert^2)
\rvert^2 d\xi_1 d \xi_2
\end{equation*}
Let $f : \R^4 \to \R^3$ be given by
$\R^2 \times \R^2 \ni (\xi_1, \xi_2)
\mapsto (\xi_1 - \xi_2, \lvert \xi_1 \rvert^2 - \lvert \xi_2 \rvert^2)
=: (\xi, \tau) \in \R^2 \times \R$.
The differential corresponding to this change of coordinates is
\begin{equation*}
df
=
\begin{bmatrix}
1 & 0 & -1 & 0 \\
0 & 1 & 0 & -1 \\
2 \xi_1^{(1)} & 2 \xi_1^{(2)} & -2 \xi_2^{(1)} & -2 \xi_2^{(2)}
\end{bmatrix}
\end{equation*}
The size $\lvert J_3 f \rvert$ of the 3-dimensional Jacobian of $f$
is  defined to be the square root of the sum
of the squares of the determinants of the $3 \times 3$ minors
of the differential $df$:
\begin{equation*}
\lvert J_3 f \rvert
:=
2\sqrt{2} \left( (\xi_2^{(2)} - \xi_1^{(2)})^2 + (\xi_2^{(1)} - \xi_1^{(1)})^2
+ (\xi_2^{(2)} - \xi_1^{(2)})^2 + (\xi_1^{(1)} - \xi_1^{(1)})^2 \right)^{1/2}
\end{equation*}
Hence
\begin{equation}
\lvert J_3 f \rvert = C \lvert \xi_2 - \xi_1 \rvert \geq C \dist(\Omega_1, \Omega_2)
\label{J3}
\end{equation}
By the coarea formula (see \cite[\S 3]{Fe59}),
\begin{align}
G &=
\int_{(\xi_1, \xi_2) \in \Omega}
\lvert
g(\xi_1 - \xi_2, \lvert \xi_1 \rvert^2 - \lvert \xi_2 \rvert^2)
\rvert^2
d\xi_1 d \xi_2 \nonumber \\
&=
\int_{\xi, \tau}
\int_{ \substack{(\xi_1, \xi_2) \in \Omega : \\
\xi = \xi_1 - \xi_2 \\ \tau = \lvert \xi_1 \rvert^2 - \lvert \xi_2 \rvert^2}}
\lvert
g(\xi_1 - \xi_2, \lvert \xi_1 \rvert^2 - \lvert \xi_2 \rvert^2)
\rvert^2
\lvert J_3 f \rvert^{-1}(\xi_1, \xi_2)
d\mathcal{H}^1(\xi_1, \xi_2) d\xi d\tau \nonumber \\
&\leq
\int_{\xi, \tau}
\lvert g(\xi, \tau) \rvert^2
\int_{ \substack{(\xi_1, \xi_2) \in \Omega : \\
\xi = \xi_1 - \xi_2 \\ \tau = \lvert \xi_1 \rvert^2 - \lvert \xi_2 \rvert^2}}
\lvert J_3 f \rvert^{-1}(\xi_1, \xi_2)
d\mathcal{H}^1(\xi_1, \xi_2) d\xi d\tau \nonumber \\
&\leq
\int_{\xi, \tau}
\lvert g(\xi, \tau) \rvert^2 d\xi d\tau
\cdot
\sup_{\xi, \tau}
\int_{ \substack{(\xi_1, \xi_2) \in \Omega : \\
\xi = \xi_1 - \xi_2 \\ \tau = \lvert \xi_1 \rvert^2 - \lvert \xi_2 \rvert^2}}
\lvert J_3 f \rvert^{-1}(\xi_1, \xi_2)
d\mathcal{H}^1(\xi_1, \xi_2)
\label{G}
\end{align}
In view of (\ref{J3}), the right hand side of $(\ref{G})$ is bounded
(up to a constant) by
\begin{equation*}
\lVert g \rVert_{L^2}^2 \cdot \dist(\Omega_1, \Omega_2)^{-1}
\cdot
\sup_{\xi, \tau}
\int_{ \substack{\xi = \xi_1 - \xi_2 \\ \tau = \lvert \xi_1 \rvert^2 - \lvert \xi_2 \rvert^2}}
\chi_{\Omega}(\xi_1, \xi_2)
d\mathcal{H}^1(\xi_1, \xi_2)
\end{equation*}
\end{proof}

A straightforward application of Lemma \ref{L:BilinearEstimate} yields
\begin{cor}[Bourgain's improved bilinear Strichartz estimate \cite{Bo98}]\label{C:Bourgain}
a) Let $\mu, \lambda$ be dyadic frequencies, $\mu \ll \lambda$. Let $\phi_\mu, \psi_\lambda$
denote free waves respectively localized in frequency to $I_\mu$ and $I_\lambda$. Then
\begin{equation}
\lVert \bar{\phi}_\mu \psi_\lambda \rVert_{L^2} \lesssim \frac{\mu^{1/2}}{\lambda^{1/2}}
\lVert \phi_\mu(0) \rVert_{L^2_x}
\lVert \psi_\lambda(0) \rVert_{L^2_x}
\label{Bourgain}
\end{equation}
b) If either $\phi_\mu$ or $\psi_\lambda$ is further frequency
localized to a box of size $\alpha\times\alpha$, then we have the
better estimate
\begin{equation}
\lVert \bar{\phi}_\mu \psi_\lambda \rVert_{L^2} \lesssim \frac{\alpha^{1/2}}{\lambda^{1/2}}
\lVert \phi_\mu(0) \rVert_{L^2_x}
\lVert \psi_\lambda(0) \rVert_{L^2_x}
\label{Bourgain2}
\end{equation}
\end{cor}

As a corollary of the proof of Lemma \ref{L:BilinearEstimate}, we obtain the following.
\begin{cor}\label{C:BilinearEstimate}
Let $u(x, t) = e^{it \Delta}u_0(x), v(x, s) = e^{i s
\Delta}v_0(x)$, where $u_0, v_0 \in L^2(\R^2)$. Let $\Omega_1$
denote the support of $\hat{u}_0(\xi_1)$, $\Omega_2$ the support
of $\hat{v}_0(\xi_2)$. Assume that for all $\xi_1 \in \Omega_1$
and $\xi_2 \in \Omega_2$ we have
\begin{equation*}
| \xi_1 \wedge \xi_2 | \sim \beta
\end{equation*}
Then
\begin{equation}\label{BilinearEstimate}
\lVert u \bar{v} \rVert_{L^2_{s, t, x}}
\lesssim
\beta^{-1/2}
\lVert u_0 \rVert_{L^2} \lVert v_0 \rVert_{L^2}
\end{equation}
\end{cor}
\begin{proof}
As in the proof of Lemma \ref{L:BilinearEstimate}, we use a duality argument.
The key is to bound
\[
\int_{(\xi_1, \xi_2) \in \Omega} \lvert g(\xi_1 - \xi_2, | \xi_1 |^2, | \xi_2 |^2) \rvert d\xi_1 d\xi_2
\]
in $L^2$. In this setting, the proof is simpler because the change of variables $f$ is given by
$\R^2 \times \R^2 \ni (\xi_1, \xi_2) \mapsto (\xi_1 - \xi_2, |\xi_1|^2, |\xi_2|^2) \in \R^2 \times \R \times \R$
so that $f: \R^4 \to \R^4$ and $| df | \sim | \xi_1 \wedge \xi_1 |$.
\end{proof}

In order to achieve a gain at matched frequencies, we localize the output in both
frequency and modulation, seeking to bound
$P_{\mu} Q_{\nu} (\bar{\phi}_\lambda \psi_\lambda)$ in $L^2$.
That Lemma \ref{L:BilinearEstimate} may be used efficiently,
we introduce an adapted frequency-space decomposition of annuli $I_\lambda \subset \R^2$
that depends upon both the output frequency and modulation cutoff scales $\mu$
and $\nu$.

\begin{defin}[Frequency decomposition]\label{D:FreqDeco}
Suppose $\mu, \nu, \lambda$ are dyadic frequencies satisfying $\mu \ll \lambda$
and $\nu \leq \mu \lambda$. We define a partition of
$I_\lambda$ into curved boxes as follows. First, partition $I_\lambda$
into $\lambda^2 / \nu$ annuli of equal thickness. Next, uniformly
partition the annuli into $\lambda / \mu$ sectors of equal angle. The resulting set of curved
boxes we call $\cQ = \cQ(\mu, \nu, \lambda)$.
\end{defin}

The curved sides of the boxes in $\cQ$ have
length $\sim \mu$, whereas the straight sides of the boxes have length $\sim \nu / \lambda$.
By adapting a suitable partition of unity to the decomposition, we have
\[
f = \sum_{ \substack{\mu \ll \lambda \\ \nu \leq \mu \lambda}} \sum_{R \in \cQ(\mu, \nu, \lambda)} P_{R} f
\]

Note that we may extend $\cQ(\mu, \nu, \lambda)$ to all smaller dyadic
scales $\lambda^\prime < \lambda$ in the following way: take the
partition $\cQ(\mu, \nu, \lambda^\prime)$ and cut the annuli into
$\lambda / \lambda^\prime$ smaller annuli of equal thickness. In this
way we can impose a finer scale on lower frequencies.

\begin{cor}\label{C:bilinearL2}
  Let $\mu, \nu, \lambda$ be dyadic frequencies satisfying $\mu
  \lesssim \lambda$ and $\nu \lesssim \mu \lambda$.  Let
  $\phi_\lambda, \psi_\lambda$ be free waves with frequency support
  contained in $I_\lambda$.  Then
\begin{equation}\label{bilinearL2}
\lVert P_{\mu} \qo_{\nu} (\bar{\phi}_\lambda \psi_\lambda) \rVert_{L^2}
\lesssim
\frac{\nu^{1/2}}{(\mu \lambda)^{1/2}}
\lVert \phi_\lambda \rVert_{L^2_x}
\lVert \psi_\lambda \rVert_{L^2_x}
\end{equation}
\end{cor}
\begin{proof}
The frequency restriction $P_{\mu}$ applied to $P_R \bar{\phi}_\lambda P_{R^\prime} \psi_\lambda$
restricts us to looking at the subcollection of boxes $R, R^\prime \in \cQ$
separated by a distance $\sim \mu$. This subcollection is further restricted by
the temporal frequency multiplier $\qo_{\nu}$.
Let $\xi_1 \in R, \xi_2 \in R^\prime$.
The modulation $\tau$ of the product
$P_R \bar{\phi}_\lambda P_{R^\prime} \psi_\lambda$ is given by
\[
| \xi_1 |^2 - | \xi_2 |^2
=
(\xi_1 - \xi_2) \cdot (\xi_1 + \xi_2)
\]
Because we apply $P_\mu$, $| \xi_1 - \xi_2 | \sim \mu$,
and therefore necessarily $\tau$ lies in the range $ |\tau| \lesssim \mu \lambda$.
We write $|\tau| \sim \mu \lambda \cos \theta$, where $\theta$ is the angle between
$\xi_1 - \xi_2$ and $\xi_1 + \xi_2$.
Applying $Q_{\nu}$ restricts $\tau$ so that
$\lvert \tau \rvert \sim \nu$ and in particular
$\lvert \cos \theta \rvert \sim \nu / (\mu \lambda)$.

These restrictions motivate defining the set of interacting
  pairs of boxes $\cP = \cP(\mu, \nu, \lambda)$ as the collection of all pairs
$(R, R^\prime) \in \cQ \times \cQ$ ($\cQ = \cQ(\mu, \nu,
\lambda)$) for which all $(\xi_1,
\xi_2) \in R \times R^\prime$ satisfy $| |\xi_1|^2 - |\xi_2|^2 | \sim
\mu$ and $|\xi_1 + \xi_2| \sim \nu$.

Note that, for $R \in \cQ(\mu, \nu, \lambda)$ fixed, the number $p$ of
interacting pairs of boxes $P \in \cP(\mu, \nu, \lambda)$ containing $R$ is
O(1) uniformly in $\mu, \nu, \lambda$.  This is a consequence of the
restrictions $\lvert \cos \theta \rvert \sim \nu / (\mu \lambda)$ and
$\lvert \xi \rvert \sim \mu$: they jointly enforce at most $O(1)$
translations of a distance $\sim \nu / \lambda$, which is precisely
the scale of the short sides of the boxes.
%In other words, if one box in a pair is taken as fixed, then
%the positional uncertainty in frequency space of the remaining box induced by the cutoffs
%coincides with the dimensions of the box.

It remains only to show that for $(R, R^\prime) \in \cP$ we have
\begin{equation}
\sup_{\xi, \tau}
\int_{ \substack{\xi = \xi_1 - \xi_2 \\ \tau = \lvert \xi_1 \rvert^2 - \lvert \xi_2 \rvert^2}}
\chi_{R}(\xi_1) \chi_{R^\prime}(\xi_2)
d\mathcal{H}^1(\xi_1, \xi_2)
\lesssim \frac{\nu}{\lambda}
\end{equation}
Fix $\xi \in \R^2, \tau \in \R$, $\xi \neq 0$,
and consider the constraint equations
\begin{equation}
\begin{cases}
\xi &= \xi_1 - \xi_2 \\
\tau &= \lvert \xi_1 \rvert^2 - \lvert \xi_2 \rvert^2
\end{cases}
\end{equation}
These determine a line in $\R^2$:
\begin{equation*}
\tau = (\xi_1 - \xi_2) \cdot (\xi_1 + \xi_2) = -\xi \cdot (\xi - 2 \xi_1)
\end{equation*}
Suppose this line intersects $R$.
The angle $\rho$ that it forms with the long side length of $R$
satisfies $\lvert \cos \rho \rvert \sim \nu / (\mu \lambda)$ due to the modulation
constraint (note that at the scale of these boxes, the effects of curvature can be neglected).
Since the long side of $R$ has length $\sim \mu$
and the short side length $\sim \nu / \lambda$, it follows that the
total intersection length is $O(\nu / \lambda)$.
\end{proof}

\subsection{Extensions}
Now we extend the bilinear estimates to $U^2_\Delta$ and $\due$ functions; 
since these extensions are valid for both spaces, we simplify the notation to $U^2$.

Our first application of this proposition is in observing that (\ref{Bourgain}) of Corollary \ref{C:Bourgain}
extends to $U^2$ functions. This follows easily from the atomic decomposition.
\begin{cor}\label{C:BourgainExt}
Let $\phi_\mu, \psi_\lambda \in U^2$ be respectively
localized in frequency to $I_\mu$ and $I_\lambda$,
$\mu\ll\lambda$. Then
\begin{equation}\label{BourgainExt}
\lVert \bar{\phi}_\mu \psi_\lambda \rVert_{L^2} \lesssim
\frac{\mu^{1/2}}{\lambda^{1/2}} \lVert \phi_\mu
\rVert_{U^2} \lVert \psi_\lambda \rVert_{U^2}
\end{equation}
\end{cor}

We may similarly conclude the following.
\begin{cor}\label{C:BilinearEstimateExt}
Let $\phi_1, \phi_2 \in U^2$ be respectively localized in frequency to
$\Omega_1$ and $\Omega_2$, where $\Omega_1, \Omega_2 \subset I_\lambda$.
Assume that for all $\xi_1 \in \Omega_1$ and $\xi_1 \in \Omega_2$ we have
\begin{equation*}
| \xi_1 \wedge \xi_2 | \sim \beta.
\end{equation*}
Then
\begin{equation}\label{BilinearEstimateExt}
\lVert \phi_1 \bar{\phi}_2 \rVert_{L^2_{s, t, x}} \lesssim \beta^{-1/2}
\lVert \phi_1 \rVert_{U^2} \lVert \phi_2 \rVert_{U^2}
\end{equation}
\end{cor}

\begin{lem}\label{L:QmodL2}
Let $Q_1, Q_2 \in \{ Q_{\leq \nu_1}, Q_{\nu_2}, Q_{\geq \nu_3}, 1 : \nu_1, \nu_2, \nu_3 \; \text{dyadic}\}$.
Let $\phi_\mu, \phi_\lambda \in U^2 $ have respective frequency supports contained in
$\alpha$ boxes lying in $I_\mu$ and $I_\lambda$, where $\mu \ll \lambda$. Then
\begin{equation}\label{QmodL2}
\lVert Q_1 \phi_\mu \cdot Q_2 \phi_\lambda \rVert_{L^2} \lesssim
\frac{\alpha^{1/2}}{\lambda^{1/2}} \lVert \phi_\mu
\rVert_{U^2} \lVert \phi_\lambda \rVert_{U^2}
\end{equation}
\end{lem}
\begin{proof}
First consider the case where $Q_1 = 1$ and $Q_2$ is of the form $Q_{\leq \nu}$.
As $Q_2$ is a Fourier multiplier with (Schwartz) symbol
\[
b(\xi, \tau) := \chi((\tau - |\xi|^2)/\nu)
\]
we have
\[
Q_2 \phi_\lambda(x, t) = (\tilde{b} * \phi_\lambda)(x, t)
= \int \tilde{b}(y, s) \phi_\lambda(x - y, t - s) dy ds
\]
and so it follows that the left hand side of (\ref{QmodL2}) admits the representation
\[
\lVert \phi_\mu(x, t) \int \tilde{b}(y, s) \phi_\lambda(x - y, t - s) dy ds \rVert_{L^2_{x, t}}
\]
Suppose we freeze $y, s$ and consider
\[
\lVert \phi_\mu(x, t) \tilde{\phi}_\lambda(x - y, t - s) \rVert_{L^2_{x, t}}
\]
By replacing $\phi_\mu$ and the translated $\phi_\lambda$ with atoms,
we obtain by Lemma \ref{L:BilinearEstimate} and the fact that
the $U^2 $ spaces are translation invariant that
\[
\lVert \phi_\mu(x, t) \tilde{\phi}_\lambda(x - y, t - s)
\rVert_{L^2_{x, t}} \lesssim \frac{\alpha^{1/2}}{\lambda^{1/2}}
\lVert \phi_\mu \rVert_{U^2} \lVert \phi_\lambda
\rVert_{U^2}
\]
Since $\tilde{b}(x, t)$ is integrable with bound independent of $\nu$,
(\ref{QmodL2}) follows in this special case.

This argument clearly generalizes to
$Q_1, Q_2 \in \{ Q_{\leq \nu_1}, Q_{\nu_2}, 1 : \nu_1, \nu_2 \; \text{dyadic}\}$.
In order to accommodate $Q_{\geq \nu}$, we apply the above argument
to $1 - Q_{\geq \nu}$ and use the triangle inequality.
\end{proof}

\begin{lem}\label{bilinearUV}
Let $\phi_\mu, \phi_\lambda \in \ell^2 V^{2, \sharp}$
be respectively frequency localized to dyadic scales $\mu, \lambda$. Then
\begin{equation}\label{eq:bilinearUV}
\|\phi_\mu \phi_\lambda\|_{L^2}\lesssim
\frac{\mu^{1/2}}{\lambda^{1/2}}
(\log \mu )^2
\|\phi_\mu\|_{l^2 V^{2,\#}}\|\phi_\lambda\|_{V^{2,\#}}
\end{equation}
b) If either $\phi_\mu$ or $\psi_\lambda$ is further frequency
localized to space-time boxes of respective size $\alpha\times\alpha \times
\alpha \mu$ or $\alpha\times\alpha \times
\alpha \lambda$, then we have the better estimate
\begin{equation}\label{eq:bilinearUV1}
\|\phi_\mu \phi_\lambda\|_{L^2}\lesssim
\frac{\alpha^{1/2}}{\lambda^{1/2}}
(\log \mu )^2
\|\phi_\mu\|_{l^2 V^{2,\#}}\|\phi_\lambda\|_{V^{2,\#}}
\end{equation}
\end{lem}
\begin{proof}
We first easily dispense with the high modulations ( $\gtrsim \mu \lambda$)
in $\phi_\lambda$ by using the trivial pointwise bound for $\phi_\mu $. Once $\phi_\lambda$
is restricted to small modulations we localize it to small angles. On the other hand,
for $\phi_\mu$ we take advantage of the $l^2$ structure to localize the estimate
to a cube of size $\mu \times \mu \times 1$, using a smooth cutoff $\chi_\mu$.
Thus it remains to estimate the expression $\chi_\mu \phi_\mu P_\e \phi_\lambda$.

Our starting point is the estimate (\ref{BourgainExt}) from Corollary
\ref{C:BourgainExt}, which yields the correct bound for $\phi_\lambda$
in $\due$ and $\phi_\mu$ in $U^2_\Delta$.
Next we put $\phi_\lambda$ in $|D|^{-\frac{1}{2}}U^p_\e$, which embeds into $\lambda^{-1/2}
L^{\infty, 2}_\e$, and $\phi_\mu$  in $V^p_\Delta$, which has the property that
$ \chi P_\mu V^p_\Delta \subset \chi \mu^2 L^\infty \subset   \mu^\frac52  L^{2, \infty}_\e$.
 Therefore by Lemma
\ref{lem:interpol}
\[
\lVert P_\e \phi_\lambda \chi_\mu \phi_\mu \rVert_{L^2}
\lesssim
\frac{\mu^{1/2}}{\lambda^{1/2}}
(\log \mu^{N-1/2} + 1)^2
\lVert \phi_\lambda \rVert_{V^2_\e} \lVert \phi_\mu \rVert_{V^2_\Delta}
\]

\end{proof}

\begin{lem}\label{C:bilinearL2-UV}
  Let $\mu, \nu, \lambda$ be dyadic frequencies satisfying $\mu
  \ll \lambda$ and $\mu^2 \lesssim \nu \lesssim \mu \lambda$.  Let
  $\phi_\lambda, \psi_\lambda$ be 
  %free 
  waves with frequency support
  contained in $I_\lambda$.  Then
\begin{equation}\label{L2biUV}
\lVert P_{\mu} \qo_{\nu} (\overline{P_\e Q_{\lesssim \mu \lambda} \phi}_\lambda \cdot
P_\e Q_{\lesssim \mu \lambda} \psi_\lambda)
\rVert_{L^2}
\lesssim
\frac{\nu^{1/2}}{(\mu \lambda)^{1/2}}
\lVert \phi_\lambda \rVert_{\due}
\lVert \psi_\lambda \rVert_{\due}
\end{equation}
where $\qo_\nu$ is replaced by $\qo_{\lesssim \mu^2}$ if $\nu
\leq \mu^2$.
\end{lem}
\begin{proof}
Notice that at frequency $\lambda$, the characteristic surface has slope $\lambda$. Once we apply the modulation localization $Q_{\lesssim\mu\lambda}$, $\phi_{\lambda}$ and $\psi_{\lambda}$ will be localized horizontally to boxes of size $\mu\times\mu$.

Using the square summability in Lemma~\ref{U2:sqsum} with respect to vertical rectangles
of size $\mu \times \nu$ in the $\e^\perp$ plane, 
and then the modulation localization,
the problem reduces  to the case of free waves which are localized in horizontally
and vertically separated cubes of size $\mu \times \mu \times \nu$. Then we can drop
both  $P_\mu$ and $\qo_\nu$, and apply the localized version of Lemma~\ref{C:bilinearL2}
extended to $U^2_\e$ atoms.

\end{proof}

%%%%%%%%%%%%%%%%%%%%%%%%%%%%%%%%%%%%%%%%%%%%%%%%%%%%%%%
%%%%%%%%%%                                                      
%%%%%%%%%%                      Estimates on A, B               
%%%%%%%%%%                                                    
%%%%%%%%%%%%%%%%%%%%%%%%%%%%%%%%%%%%%%%%%%%%%%%%%%%%%%%
\section{Estimates on $A$, $B$}

In this section we consider the system \eqref{A1A2Initial} for $A_1$
and $A_2$ under the assumption that $\phi$ is small in the space
$X^s$.  Unfortunately, it does not seem possible to directly obtain
estimates for $A_1,A_2$ which suffice in order to close the rest of
the argument. For this reason, we express $A_1,A_2$ in terms of the
milder variable $B_1,B_2$, which we can estimate in better norms and
treat perturbatively. Thus we end up solving the system \eqref{B1B2}
perturbatively.

We begin with estimates for the first two components of $A_1$, $A_2$ in
\eqref{A1A2Initial}.

\begin{lem}[Linear flow estimate] Assume that $\phi_0 \in H^s$. Then we have
\begin{equation} \label{e:aFlow}
\left<D\right>^sH^{-1}A_x(0)\in
L^{4}_{x,t}[0,1]\cap L^{2,6}_{\e}[0,1]
\end{equation}
\end{lem}
\begin{proof}
We recall that $A_{1}(0) = \frac12 \Delta^{-1}  \partial_2 |\phi_0|^2$ and
$A_{2}(0) = - \frac12 \Delta^{-1}  \partial_1 |\phi_0|^2$. We split
the data into low and high frequency components. For the high frequency part
we have the multiplicative estimate
\[
\|\left<D\right>^s P_{>0} A_x(0)\|_{L^2} =
\|\left<D\right>^s \Delta^{-1}  \partial_x |\phi_0|^2 \|_{L^2} \lesssim \| \phi_0\|_{H^s}^2
\]
By parabolic regularity this gives
\[
\| \left<D\right>^s P_{>0} H^{-1} A_x(0)\|_{L^2_t \dot H^1_x \cap L^2_x L^\infty_t}
\lesssim \| \phi_0\|_{H^s}^2
\]
and the conclusion follows by Sobolev embeddings.

For the low frequency part  we can only use the straightforward estimate
\[
\||u_0|^2\|_{L^1} \lesssim \|u_0\|_{H^s}^2
\]
We write
\[
P_{<0} H^{-1} A_x(0) = K(t) \ast |\phi_0|^2
\]
where the kernel $K(t)$ of $\Delta^{-1} \partial_x H^{-1}P_{<0} $ is given by
\[
K(t)=\mathcal{F}^{-1}({\psi(|\xi|)\frac{\xi_j}{|\xi|^2}} e^{-t |\xi|^2})
\]
and  $\psi(\xi)$ is the symbol of the
Littlewood-Paley projection $P_{<0}$.
For $K(t,x)$,  we notice the following two facts:
\begin{itemize}
\item $|K(t,x)|\lesssim 1$. This is because
\[|K(t,x)|\lesssim \int\frac{|\xi_j|}{|\xi|^2}\psi(|\xi|)|\xi|d|\xi|\lesssim \int\psi(r)dr\]

\item  $K(t,x)\sim \frac{x_j}{|x|^2} $ as $x\rightarrow \infty$. In
radial coordinates, $(\xi_1,\xi_2)=(r\cos\theta, r\sin\theta)$,
and we can assume $\xi_j$ to be $\xi_1$. We write
\begin{eqnarray*} K(t,x)&=& \int_0^{2\pi}\int_0^\infty
\frac{ \psi(r)r\cos\theta}{r^2}e^{i r(x_1\cos\theta+x_2\sin\theta)}e^{-tr^2} rdr d\theta\\
&=& \int_0^{2\pi}\int_0^\infty
 \psi(r)\cos\theta e^{i|x| r(\frac{x_1}{|x|}\cos\theta+ \frac{x_2}{|x|}\sin\theta)}e^{-tr^2} dr d\theta
\end{eqnarray*}
By the method of stationary phase, we get the asymptotic
$\frac{x_j}{|x|^2}$.
\end{itemize}
Thus we have the uniform bound
\[
|K(t,x)| \lesssim \langle x \rangle^{-1}
\]
With this, we conclude that
\[
\|P_{<0} H^{-1} A_x(0)\|_{L^4_x}\lesssim
\|K(t,x)\|_{L^4}\||\phi_0|^2\|_{L^1_x}\lesssim \|u_0\|_{H^s}^2
\]
\[
\|P_{<0} H^{-1} A_x(0)\|_{L^{2,6}_{\textbf{e}}}\lesssim
\|K(t,x)\|_{L^{2,6}_{\textbf{e}}}\||\phi_0|^2\|_{L^1_x} \lesssim
\|u_0\|_{H^s}^2
\]
and the proof of the lemma is concluded.
\end{proof}

\begin{lem} \label{L:AxAngStrichartz}
We have the following bounds for the Littlewood-Paley pieces of
$H^{-1}(\bar{\phi} \, \partial_x\phi)$:
\begin{equation}\label{AxAng1}
\|H^{-1}P_{\lambda_3}(\bar{\phi}_{\lambda_1}\partial_x\phi_{\lambda_2})
\|_{\sum_\e L^{\infty,3}_\e}\lesssim
\|\phi_{\lambda_1}\|_{V^{2,\#}}\|\phi_{\lambda_2}\|_{V^{2,\#}},
 \qquad \lambda_1\sim\lambda_2\gg\lambda_3
\end{equation}
\begin{equation}\label{AxAng2}
\|H^{-1}P_{\lambda_3}(\bar{\phi}_{\lambda_1}\partial_x\phi_{\lambda_2})\|_{H^{-\frac{1}{2}}L^2_{x,t}}\! \! \lesssim
\|\phi_{\lambda_1}\|_{V^{2,\#}}\|\phi_{\lambda_2}\|_{V^{2,\#}},
\qquad \lambda_3 \sim\max\{\lambda_1, \lambda_2 \}
\end{equation}
In addition, if  $\lambda_3 \ll  \lambda_1\sim \lambda_2$, then
\begin{equation}\label{highmod-bi}
 \|H^{-1}P_{\lambda_3}\left(\bar{\phi}_{\lambda_1}\partial_x\phi_{\lambda_2}
- \qo_{\lesssim \lambda_1 \lambda_3}
(\overline{ Q_{\lesssim \lambda_1 \lambda_3}\phi}_{\lambda_1}\partial_x Q_{\lesssim \lambda_1 \lambda_3}\phi_{\lambda_2})\right)\| _{H^{-\frac{1}{2}}L^2_{x,t}}\! \! \lesssim
\|\phi_{\lambda_1}\|_{V^{2,\#}}\|\phi_{\lambda_2}\|_{V^{2,\#}}
\end{equation}
%Similar estimates hold for $H^{-1}(\phi \, \partial_x\bar \phi)$.
\end{lem}

\begin{proof}
  For the first estimate we first dispense with the case when either
  factor has high modulation. Indeed, assume the first factor has high
  modulation. Then we have
\[
\|Q_{\gtrsim \lambda_1} \bar{\phi}_{\lambda_1}\partial_x\phi_{\lambda_2}\|_{L^{\frac43}}
\lesssim \|Q_{\gtrsim \lambda_1} \bar{\phi}_{\lambda_1}\|_{L^2}
\|\partial_x\phi_{\lambda_2}\|_{L^{4}} \lesssim
\|\phi_{\lambda_1}\|_{V^{2,\#}}\|\phi_{\lambda_2}\|_{V^{2,\#}}
\]
and the conclusion follows due to the parabolic Sobolev embedding
\[
P_\lambda H^{-1} : L^{\frac43} \to L^{\infty,3}_\textbf{e}
\]
Next we assume both factors have small modulations and localize the
two factors to small close angular sectors.  This is possible since
the output frequency is much lower. Then we choose a common admissible
direction $\textbf e$, and bound both factors in
$L^{\infty,2}_\textbf{e}$,
\[
\| P_\textbf{e} \bar{\phi}_{\lambda_1}\partial_x P_\textbf{e}\phi_{\lambda_2}\|_{
L^{\infty,1}_\textbf{e}} \lesssim  \|\phi_{\lambda_1}\|_{V^{2,\#}}\|\phi_{\lambda_2}\|_{V^{2,\#}}
\]
Then the conclusion follows due to the parabolic Sobolev embedding
\[
P_\lambda H^{-1} :  L^{\infty,1}_\textbf{e} \to L^{\infty,3}_\textbf{e}
\]

The second estimate \eqref{AxAng2}  is easier. We bound both factors in $L^4$ and use
the parabolic Sobolev embedding
\[
\lambda P_\lambda H^{-\frac12} :  L^2 \to L^2
\]

For the third estimate \eqref{highmod-bi} we need to consider two cases for the
terms in the difference:

(i) Both inputs have high modulation. Then we need to estimate
\[
H^{-1}P_{\lambda_3}
(\overline{ Q_{\gg \lambda_1 \lambda_3}\phi}_{\lambda_1}\partial_x
Q_{\gg \lambda_1 \lambda_3}\phi_{\lambda_2})
\]
Placing the two factors in $L^2$, we conclude using the bound
\[
H^{-\frac12}P_{\lambda_3}: L^1 \to \lambda_3 L^2
\]

(ii) One input and the output have high modulation. Then we need to estimate
\[
H^{-1}P_{\lambda_3} \qo_{\gg \lambda_1 \lambda_3}
(\overline{ Q_{\gg \lambda_1 \lambda_3}\phi}_{\lambda_1}\partial_x
\phi_{\lambda_2})
\]
Placing the first factor in $L^2$ and the second in $L^\infty
L^2$, we conclude using the bound
\[
H^{-\frac12}P_{\lambda_3}\qo_{\gg \lambda_1 \lambda_3}: L^2L^1
\to (\lambda_3/\lambda_1 )^\frac12 L^2
\]
\end{proof}

\begin{lem}
Assume that $\phi$ is small in $X^s$. Then the system \eqref{B1B2}
admits a unique solution $B_1,B_2$, depending smoothly on $\phi$, and
with regularity
\begin{equation} \label{e:B}
 \la D\ra^{s} B \in H^{-\frac{1}{2}}L^2_{x,t}[0,1]
\end{equation}

\end{lem}

\begin{rem}
Notice the following Sobolev embeddings hold:
\begin{align}
H^{-1}(L^{2, 6/5}_{\mathbf{e}})
&\hookrightarrow H^{-{\frac{1}{2}}}L^2_{t,x} \label{Sobem1}\\
H^{-{\frac{1}{2}}}L^2_{t,x}
&\hookrightarrow L^{4}_{t,x} \label{Sobem2} \\
H^{-{\frac{1}{2}}}L^2_{t,x} &\hookrightarrow L^{2,6}_{\mathbf{e}}
\label{Sobem3}
\end{align}
So $\left<D\right>^sH^{-1}(A_x \lvert \phi \rvert^2)\in
L^{4}_{t,x}\cap L^{2,6}_{\mathbf{e}}. $ This is also true for
$\left<D\right>^sH^{-1}P_{\lambda_3}(\bar{\phi}_{\lambda_1}\partial_x\phi_{\lambda_2})$
when $\lambda_3 \sim\max(\lambda_1, \lambda_2 )$.
\end{rem}

\begin{proof}
Let us recall the formula \eqref{B1B2} for $B$
\begin{equation*} 
\begin{split}
B_1 &=  H^{-1}( (H^{-1}A_2( 0) +H^{-1} [\Re(\bar{\phi}
\partial_1 \phi) + \Im(\bar{\phi}
\partial_1 \phi) ])|\phi|^2) + H^{-1} (B_2 |\phi|^2)
\\
B_2 & =  H^{-1}( (H^{-1}A_1( 0)-H^{-1} [\Re(\bar{\phi} \partial_2 \phi)+\Im(\bar{\phi} \partial_2\phi)] )|\phi|^2) -  H^{-1} (B_1 |\phi|^2)
\end{split}
\end{equation*}

We use the contraction principle to solve for $B$ in the space in the lemma.
We begin with the first term on the right. By \eqref{e:aFlow} we have
$ \la D\ra^{s} H^{-1}A_x(0)\in L^4_{x,t}[0,1]$.  Also by virtue of the Strichartz estimate
(\ref{est:UpStrichartz}), we have $\left<D\right>^{s}\phi \in
L^4_{x,t}$.
Then we obtain
\[
 \la D\ra^{s} H^{-1}\left( H^{-1}A_x(0) |\phi|^2 \right) \in H^{-1} L^{\frac43}
\subset H^{-\frac{1}{2}}L^2_{x,t}[0,1]
\]

By the second part of Lemma~\ref{L:AxAngStrichartz} and the embedding
$H^{-\frac12} L^2 \subset L^4$,   the same argument applies for the
$\mathrm{low} \times \mathrm{high} \to \mathrm{high}$ contribution in the second term in the  formula above,
as well as for the third term.

For the $\mathrm{high} \times \mathrm{high} \to \mathrm{low}$ contribution we need to use the first
 part of Lemma~\ref{L:AxAngStrichartz}, and so we only have
the $L^{\infty,3}_\textbf{e}$ norm available; however, a redeeming feature is that we
obtain $l^2$ dyadic summation. We also have  $l^2$ dyadic summation in the
$L^4$ bound for $ \la D\ra^{s} \phi$, and this is still sufficient in view of the embedding
\[
P_\lambda H^{-\frac12} : L^{2,\frac{6}5}_\textbf{e} \to L^2
\]
\end{proof}

%%%%%%%%%%%%%%%%%%%%%%%%%%%%%%%%%%%%%%%%%%%%%%%%%%%%%%%
%%%%%%%%%%                                                    
%%%%%%%%%%                      Cubic terms            
%%%%%%%%%%                                                     
%%%%%%%%%%%%%%%%%%%%%%%%%%%%%%%%%%%%%%%%%%%%%%%%%%%%%%%
\section{Cubic terms}
\label{s:cubic}

In this section, we focus on controlling the main cubic terms of the nonlinearity,
i.e., \eqref{N31}, \eqref{N32}, and \eqref{N33}.
The term $N_{3,3}$ is controlled using the $L^4$ Strichartz estimate.
Controlling $N_{3,1}$ and $N_{3,2}$ requires considerable additional work.
Our strategy is as follows. By duality, we can rephrase the estimate for $N_{3,1}$
as a quadrilinear estimate
\[
\left |\int N_{3,1}(\phi^{(1)},\phi^{(2)},\phi^{(3)},\phi^{(4)})\ dx dt \right|
\lesssim \|\phi^{(1)}\|_{X^{\#,s}} \|\phi^{(2)}\|_{X^{\#,s}} \|\phi^{(3)}\|_{X^{\#,s}}
\| \phi^{(4)}\|_{X^{-s}}
\]
where  by a slight abuse of notation we set
\[
N_{3,1}(\phi^{(1)},\phi^{(2)},\phi^{(3)},\phi^{(4)}) =
H^{-1}(\bar{\phi}^{(1)} \partial_1 \phi^{(2)}) (\bar{\phi}^{(3)} \partial_2 \phi^{(4)})
- H^{-1}(\bar{\phi}^{(1)} \partial_2 \phi^{(2)}) (\bar{\phi}^{(3)} \partial_1 \phi^{(4)})
\]
The same can be done for $N_{3,2}$. After one integration by parts, the
corresponding quadrilinear form is split into one part which is identical to the one above,
and one which is similar but with the complex conjugation on $\phi^{(4)}$ instead
of  $\phi^{(3)}$.  The arguments that follow apply equally to both cases.

Because the four input frequencies $\xi_j$ satisfy
\begin{equation}\label{freqcon}
\xi_1 - \xi_2 + \xi_3 - \xi_4 = 0,
\end{equation}
the four input frequencies $(\lambda_1,\lambda_2,\lambda_3,\lambda_4)$
are not completely independent. Two of these are essentially at the
same scale, $\lambda$, and dominate the remaining ones, whose scales
we now call $\mu_1$ and $\mu_2$.  We assume that $\mu_1 \leq \mu_2$.
Examining the balance of frequencies
coming from the $H^s$ norms in the estimate above, we have that the worst case is
when the high frequencies cancel and we only can use the low
frequency factors.  To the above we add a final frequency parameter $\alpha$
which is the $H^{-1}$ frequency. We denote by $ N_{3,1}^\alpha$ the expression obtained
from $N_{3,1}$ by replacing $H^{-1}$ by $P_\alpha H^{-1}$.

Relaxing also the $X^{\#,s}$ norms to $X^{s}$ and then retaining only the $l^2 V^{2,\#}$
component, we are left with proving the estimate
\begin{equation}\label{quad-diad}
  \left |\int N_{3,1}^\alpha(\phi^{(1)}_{\lambda_1},\phi^{(2)}_{\lambda_2},\phi^{(3)}_{\lambda_3},\phi^{(4)}_{\lambda_4})\ dx dt \right|
  \lesssim  \mu_2^s \|\phi^{(1)}_{\lambda_1}\|_{l^2 V^{2,\#}} \|\phi^{(2)}_{\lambda_2}\|_{l^2 V^{2,\#}} \|\phi^{(3)}_{\lambda_3}\|_{l^2 V^{2,\#}}
  \| \phi^{(4)}_{\lambda_4}\|_{l^2 V^{2,\#}}
\end{equation}
Here the $\lambda$ summation is produced from the $l^2$ dyadic
summations for the frequency $\lambda$ factors, while the $\mu_1$
and $\mu_2$ summations are satisfactory due to the frequency
factor above (applied with a smaller $s$). For $\alpha$ we must
have either $\alpha \approx \lambda$ or $\alpha \lesssim \mu_2$,
and so the $\alpha$ summation also yields at most $\log \mu_2$
acceptable losses.

We distinguish cases according to whether the highest two input frequencies
are balanced or unbalanced, where
we say that a pair $\phi_{\lambda_{2j-1}} \phi_{\lambda_{2j}}$ is \emph{balanced}
if $\lambda_{2j-1} \sim \lambda_{2j}$ and \emph{unbalanced} otherwise.
We consider cases in increasing order of difficulty:
\bigskip

{\bf Case 1: Unbalanced case,  $\alpha \approx \lambda$.}
In this case we must have at least one $\lambda$ frequency factor
in each of the two pairs. It suffices to consider the case where the derivatives
fall on the high-frequency $\lambda$ terms.
The null structure is not used, and so we need only bound
\begin{equation}\label{4lin-unbal}
\left\lvert \int
P_\lambda H^{-1}(\bar{\phi}_{\mu_1} \partial_1 \phi_\lambda)
(\bar{\phi}_{\mu_2} \partial_2 \phi_\lambda) \ dx dt  \right\rvert
\end{equation}

To each pair $\bar{\phi}_{\mu_j} \phi_\lambda$ we apply the bilinear estimate (\ref{bilinearUV}),
obtaining a combined bound of $(\mu_1 \mu_2)^{1/2} / \lambda$. The two derivatives
in (\ref{4lin-unbal}) are multipliers whose contribution is bounded by $\lambda^2$,
while $H^{-1}$ is a multiplier controlled here by $\lambda^{-2}$.
Therefore (\ref{4lin-unbal}) is bounded by $O(\mu_1^{1/2} \mu_2^{1/2} /\lambda)$,
and \eqref{quad-diad} follows.

\bigskip

{\bf Case 2: Balanced/unbalanced case, $\alpha \approx \mu_2 \ll
  \lambda$.}   Then we need to
consider the expression
\begin{equation}\label{4lin-unbal-a}
\left\lvert \int
P_{\mu_2} H^{-1}(\bar{\phi}_{\lambda} \partial_1 \phi_\lambda)
(\bar{\phi}_{\mu_1} \partial_2 \phi_{\mu_2}) \ dx dt  \right\rvert
\end{equation}
 We consider three cases:
\medskip

{\bf Case 2a:} One of the frequency $\lambda$ factors has high modulation
$\gtrsim \lambda^2$.  We bound that factor in $L^2$ and all others in $L^4$, using
\[
\mu P_\mu H^{-1} : L^\frac43 \to L^2
\]
Once this case is dealt with, we can localize both frequency $\lambda$ factors first to
small modulations and then  to a small angle.

\medskip
 {\bf Case 2b:} {\em Both $\phi_{\mu_1}$ and  $\phi_{\mu_2}$ have high modulations.}
Then we estimate
\[
I_{2b} =  \int
H^{-1}(P_{\e} \bar{\phi}_{\lambda} \partial_1 P_{\e} \phi_\lambda)
(Q_{> \mu_2^2} \bar{\phi}_{\mu_1} \partial_2Q_{> \mu_2^2} \phi_{\mu_2}) \ dx dt
\]
by
\[
\begin{split}
|I_{2b}|  \lesssim & \
\mu_2 \| P_{\mu_2} H^{-1}(P_{\e} \bar{\phi}_{\lambda} \partial_1 P_{\e} \phi_\lambda)\|_{L^\infty}
\|Q_{> \mu_2^2} \phi_{\mu_1} \|_{L^2} \|Q_{> \mu_2^2} \phi_{\mu_2} \|_{L^2}
\\
  \lesssim & \ \mu_2^2
 \| P_{\e} \bar{\phi}_{\lambda} \partial_1 P_{\e} \phi_\lambda)\|_{L^{\infty,1}_{\e}}
\|Q_{> \mu_2^2} \phi_{\mu_1} \|_{L^2} \|Q_{> \mu_2^2} \phi_{\mu_2} \|_{L^2}
\\
  \lesssim & \ \|\phi^{(1)}_{\lambda_1}\|_{V^{2,\#}} \|\phi^{(2)}_{\lambda_2}\|_{V^{2,\#}} \|\phi^{(3)}_{\lambda_3}\|_{V^{2,\#}}
  \| \phi^{(4)}_{\lambda_4}\|_{V^{2,\#}}
 \end{split}
\]

\medskip
{\bf Case 2c:} {\em One of $\phi_{\mu_1}$ and $\phi_{\mu_2}$ has low
modulation, say $\phi_{\mu_2}$}. We shift $H^{-1}$  to the
second product, which is localized at frequency $\mu_2$; then we 
have an expression of the form 
\[
P_{\mu_2} \bar H^{-1} (\phi_{\mu_1} Q_{\leq \mu_2^2}\phi_{\mu_2} ) 
\]
The symbol of $P_{\mu_2}\bar H^{-1} $ is smooth on the 
$\mu_2^2 \times \mu_2 \times \mu_2$ scale, which is the size 
of the frequency localization for $Q_{\leq \mu_2^2}\phi_{\mu_2}$.
Then using Fourier series in both $\xi$ and $\tau$, 
we can separate variables and replace the above expression by 
a rapidly convergent sum of the form
\[
\sum_j R_j^1 \phi_{\mu_1} R_j^2Q_{\leq \mu_2^2} \phi_{\mu_2}  
\]
where $R_j^1$ has size $(|\tau|+ \mu_2^2)^{-1}$ and dyadic regularity, 
and $ R_j^2$ is smooth on the scale of the frequency support of 
$Q_{\leq \mu_2^2}\phi_{\mu_2}$. Then we can simply discard the $ R_j^2$
factor and replace $R_j^1$ by a $\mu_2^{-2}$ factor.
Then in the product estimate
\[
I_{2c} =  \mu_2^{-2} \int
(P_{\e} \bar{\phi}_{\lambda} \partial_1 P_{\e} \phi_\lambda)
( \bar{\phi}_{\mu_1} \partial_2Q_{< \mu_2^2} \phi_{\mu_2}) \ dx dt
\]
we regroup terms and use two bilinear $L^2$ bounds (\ref{bilinearUV}) to obtain
\[
|I_{2c} | \lesssim (\mu_1/\mu_2)^\frac12 (\log \mu_2)^2 \|\phi^{(1)}_{\lambda_1}\|_{V^{2,\#}}
\|\phi^{(2)}_{\lambda_2}\|_{V^{2,\#}} \|\phi^{(3)}_{\lambda_3}\|_{V^{2,\#}}
  \| \phi^{(4)}_{\lambda_4}\|_{ V^{2,\#}}
\]
\bigskip

{\bf Case 3: Balanced case, $\alpha \ll \mu_2 \leq \lambda$.} Then we
must have $\mu_1 \approx \mu_2$; we denote both by $\mu$. We need to
estimate
\[
I_{3} = \int P_\alpha H^{-1}(\bar{\phi}_{\lambda} \partial_1
\phi_\lambda) ( \bar{\phi}_{\mu} \partial_2 \phi_{\mu}) - P_\alpha
H^{-1}(\bar{\phi}_{\lambda} \partial_2 \phi_\lambda) (
\bar{\phi}_{\mu} \partial_1 \phi_{\mu})\ dx dt
\]
The difficulty in this case is that by using linear and bilinear estimates
our losses are in terms of $\lambda$ and $\mu$, while our gains from
$H^{-1}$ are only in terms of $\alpha$. This is where our heat gauge
is most useful.  We begin by peeling off some easier high modulation cases.
The important modulation threshold for  $H^{-1}$ is $\alpha \lambda$.

\medskip

{\bf Case 3a.} {\em High ($ \gg \alpha \lambda$) modulation in  $H^{-1}$.}
This forces at least one comparable modulation in each of the pairs
of factors. Thus we need to consider expressions of the form
\[
I_{3a} =  \int
P_\alpha \qo_\nu  H^{-1}( \bar{\phi}_{\lambda} \partial_1 Q_{\gtrsim \nu}  \phi_\lambda)
( \bar{\phi}_{\mu} \partial_2 Q_{\gtrsim \nu} \phi_{\mu}) \ dx dt, \qquad \nu \gg \alpha \lambda
\]
To bound this we use $L^2$ for the high modulation factors, energy for the other two
and Bernstein at frequency $\alpha$. This gives
\[
|I_{3a}| \lesssim \frac{\alpha^2 \mu \lambda}{\nu^2}
 \|\phi^{(1)}_{\lambda_1}\|_{V^{2,\#}} \|\phi^{(2)}_{\lambda_2}\|_{V^{2,\#}} \|\phi^{(3)}_{\lambda_3}\|_{V^{2,\#}}
  \| \phi^{(4)}_{\lambda_4}\|_{V^{2,\#}}
\]
which suffices.

\medskip

{\bf Case 3b.} {\em Low ($ \lesssim \alpha \lambda$) modulation in $H^{-1}$
but high ($ \gg \alpha \lambda$) modulation in the frequency $\lambda$
factors.}  This forces at least one comparable modulation in each of
the pairs of factors.  For the frequency $\mu$ factors the high
modulations $\gtrsim \mu^2$ are easy to treat. Discarding those, we
use the relation $\alpha \ll \mu$ to localize to small angles.  Thus
we need to consider expressions of the form
\[
I_{3b} =  \int
P_\alpha \qo_{\lesssim \alpha \lambda}
  H^{-1}( \overline{ Q_{\gg \alpha \lambda} \phi}_{\lambda} \partial_1
Q_{\gg \alpha \lambda}  \phi_\lambda)
( P_\e \bar{\phi}_{\mu} \partial_2 P_\e \phi_{\mu}) \ dx dt
\]
To bound this we use $L^2$ for the high modulation factors. For the frequency $\mu$
factors we  use lateral energy with respect to the admissible direction $\e$.
This gives
\[
\begin{split}
|I_{3b}| \lesssim & \lambda  \| Q_{\gg \alpha \lambda} \phi_{\lambda}\|_{L^2}
 \| Q_{\gg \alpha \lambda} \phi_{\lambda}\|_{L^2} \| \bar H^{-1}P_\alpha ( P_\e \bar{\phi}_{\mu} \partial_2 P_\e \phi_{\mu})\|_{L^\infty}
\\
\lesssim &  \ \mu  \|\phi_{\lambda}\|_{V^{2,\#}}
\|\phi_{\lambda}\|_{V^{2,\#}} \|  P_e \bar{\phi}_{\mu}  P_\e
\phi_{\mu}\|_{L^{\infty,1}_\e}
\\
\lesssim & \
 \|\phi^{(1)}_{\lambda_1}\|_{V^{2,\#}} \|\phi^{(2)}_{\lambda_2}\|_{V^{2,\#}} \|\phi^{(3)}_{\lambda_3}\|_{V^{2,\#}}
  \| \phi^{(4)}_{\lambda_4}\|_{V^{2,\#}}
\end{split}
\]
which again suffices.

At this point we can restrict ourselves to low ($ \lesssim \alpha
\lambda$) modulations in both $H^{-1}$ and the frequency $\lambda$
factors.  The proof branches again into two cases depending on whether
$\mu \ll \lambda$ or $\mu \approx \lambda$.  These two cases have
similarities but also some significant differences. One such
difference is that in the latter case we can freely restrict ourselves to low
modulations in all four factors; this turns out to be very useful in our argument.

\medskip

{\bf Case 3c(i).} {\em Low ($ \lesssim \alpha \lambda$) modulations in
  both $H^{-1}$ and the frequency $\lambda$ factors, $\mu \ll
  \lambda$.}  Localizing to small angles in both frequency $\lambda$
factors, we need to consider the expression
\[
\begin{split}
I_{3c} =  & \ \int
P_\alpha \qo_{\lesssim \alpha \lambda}
  H^{-1}( \overline{ P_\e Q_{\lesssim \alpha \lambda} \phi}_{\lambda} \partial_1
P_\e Q_{\lesssim \alpha \lambda}  \phi_\lambda)
( \bar{\phi}_{\mu} \partial_2  \phi_{\mu})  \ dx dt
\\ & \ - \int
P_\alpha \qo_{\lesssim \alpha \lambda}
  H^{-1}( \overline{ P_\e Q_{\lesssim \alpha \lambda} \phi}_{\lambda} \partial_2
P_\e Q_{\lesssim \alpha \lambda}  \phi_\lambda)
( \bar{\phi}_{\mu} \partial_1  \phi_{\mu}) \ dx dt
\end{split}
\]
We will prove that
\begin{equation}\label{quadri}
|I_{3c}| \lesssim  (\log \mu)^4
\|\phi^{(1)}_{\lambda}\|_{V^{2,\#}} \|\phi^{(2)}_{\lambda}\|_{V^{2,\#}} \|\phi^{(3)}_{\mu}\|_{l^2 V^{2,\#}}
  \| \phi^{(4)}_{\mu}\|_{l^2 V^{2,\#}}
\end{equation}
in several steps:

\medskip

{\bf Case 3c(i), Step 1:} {\em Proof of \eqref{quadri} for free waves, no log loss.}
We use the orthogonal partitioning $\cQ(\alpha, \nu, \lambda)$ of $I_\lambda$
and $I_\mu$. Let $R_1, R_2, R_3, R_4$ be boxes belonging to this partition,
where $R_1, R_2$ are $\alpha$-separated at frequency $\lambda$
and $R_3, R_4$ are $\alpha$-separated at frequency $\mu$.
By the $L^2$-orthogonality of the partition, it suffices to estimate
\begin{equation}\label{4lin-rec}
\left\lvert
\int
H^{-1}(\bar{\phi}_{R_1} \partial_1 \phi_{R_2})
\bar{\phi}_{R_3} \partial_2 \phi_{R_4} \ dx dt
-
\int
H^{-1}(\bar{\phi}_{R_1} \partial_2 \phi_{R_2})
\bar{\phi}_{R_3} \partial_1 \phi_{R_4} \ dx dt  \right\rvert
\end{equation}
We now split into two subcases according to the strength of the null form.

\medskip

\textbf{Subcase I:}
Suppose that there is $\beta \gtrsim \alpha \lambda$ such that
$| \xi_2 \wedge \xi_4 | \sim \beta$ for all $\xi_2 \in R_2$ and $\xi_4 \in R_4$.
Then $| \xi_1 \wedge \xi_3 | \sim \beta$ for all $\xi_1 \in R_1$ and $\xi_3 \in R_3$.
Let $b(y, s)$ denote the kernel of $P_\alpha \qo_{< \alpha \lambda} H^{-1}$.
Then
\begin{equation}\label{b:bd}
|b(y, s)| \lesssim \alpha^{-2} (1+\alpha |y| + \alpha^2 |s|)^{-N}
\end{equation}
Our goal is to control
\[
\beta \left\lvert \int b(y, s)
(\bar{\phi}_{R_1}  \phi_{R_2})(x - y, t - s)
(\bar{\phi}_{R_3}  \phi_{R_4})(x, t) \
ds dt dx dy
\right\rvert
\]
This is bounded by
\[
\beta \int
\| \bar{\phi}_{R_1}(x - y, t - s) \phi_{R_4}(t,x)\|_{L^2_{t,s,x}}
\| \phi_{R_2}(x - y, t - s)\bar{\phi}_{R_3}(t,x)\|_{L^2_{t,s,x}}
\sup_s |b(y, s)| dy
\]
For the $L^2$ norm we use the bilinear estimate (\ref{BilinearEstimate}) twice,
uniformly in $y$; this yields a factor of $\beta^{-1}$. Also from \eqref{b:bd} we have
\[
\int \sup_s |b(y, s)| dy \lesssim 1
\]
Thus the conclusion follows.

\medskip

\textbf{Subcase II:} Suppose again that $| \xi_2 \wedge \xi_4 | \lesssim \alpha \lambda$
for all $\xi_2 \in R_2$ and $\xi_4 \in R_4$, but now $\mu \ll \lambda$.
Our goal is now to control
\[
\alpha \lambda   \left\lvert \int b(y, s)
(\bar{\phi}_{R_1}  \phi_{R_2})(x - y, t - s)
(\bar{\phi}_{R_3}  \phi_{R_4})(x, t) \
ds dt dx dy
\right\rvert
\]
This is bounded by
\[
\alpha \lambda  \int
\| \bar{\phi}_{R_1}(x - y, t - s) \phi_{R_4}(t,x)\|_{L^2_{t,x}}
\| \phi_{R_2}(x - y, t - s)\bar{\phi}_{R_3}(t,x)\|_{L^2_{t,x}}
 |b(y, s)| ds dy
\]
and the  argument is concluded by applying  \eqref{Bourgain} twice.

\medskip

{\bf Case 3c(i), Step 2:} {\em Proof of \eqref{quadri} for $U^2$ waves, no log loss.}
Precisely, we will show that
\begin{equation}\label{quadriU}
|I_{3c}| \lesssim
\|P_\e \phi^{(1)}_{\lambda}\|_{\due} \|P_\e \phi^{(2)}_{\lambda}\|_{\due}
\|\phi^{(3)}_{\mu}\|_{ U^{2}_\Delta}
  \| \phi^{(4)}_{\mu}\|_{U^{2}_\Delta}, \qquad \mu \ll \lambda
\end{equation}
For this we try to mimic the arguments for free waves. The symbol of
$P_{\alpha}H^{-1} \qo_{<\alpha \lambda}$ is localized in a region of
size $\alpha \times \alpha \times \alpha \lambda$. We claim we can
freely localize each of the four functions to similarly sized
frequency regions.  In the case of the $U^2_\Delta$ spaces for
frequency $\mu$ factors, only the frequency localization on the
$\alpha$ scale is used, and the square summability of the $U^2_\Delta$
norm with respect to $\alpha$ rectangles is due to
Lemma~\ref{U2:sqsum}.  In the case of the frequency $\lambda$ factors
we are using the lateral flow in the $\e$ direction. Thus by
Lemma~\ref{U2:sqsum} we obtain square summability with respect to
projectors associated to vertical rectangles of size $\alpha \times
\alpha \lambda$ in $\e^\perp$ directions. However, the $Q_{\lesssim
  \alpha \lambda}$ localization of $\phi^{(1)}$ and $\phi^{(2)}$
ensures that the above localization in $\e^\perp$ directions induces
also an $\alpha$ localization in the $\e$ direction (this is the same as in the proof of Lemma~\ref{C:bilinearL2-UV}) .Thus we have
arrived at the same situation as in \eqref{4lin-rec}.

From here on, the proof proceeds exactly as in the free case,
using the observation that the bilinear $L^2$
bounds \eqref{BilinearEstimate} and \eqref{Bourgain} extend in a
straightforward manner to $U^2$ functions.

\medskip

{\bf Case 3c(i), Step 3:} {\em Proof of \eqref{quadri} for $U^2$ $\lambda$ waves and $V^2$ $\mu$-waves, log loss.}
Precisely, we will show that
\begin{equation}\label{quadriUV}
|I_{3c}| \lesssim
(\log \mu)^2
\|\phi^{(1)}_{\lambda}\|_{\due} \|\phi^{(2)}_{\lambda}\|_{\due} \|\phi^{(3)}_{\mu}\|_{ l^2 V^{2,\sharp}}
  \| \phi^{(4)}_{\mu}\|_{l^2 V^{2,\sharp}}, \qquad \mu \ll \lambda
\end{equation}
The role of the $l^2$ structure here is to allow localization to the unit time scale.
Once this is done, we decompose
\[
\phi^{(3)}_{\mu} = Q_{< \mu^N} \phi^{(3)}_{\mu}+ Q_{> \mu^N} \phi^{(3)}_{\mu}
\]
For the first component, using the unit time localization, we have
\[
\|Q_{< \mu^N} \phi^{(3)}_{\mu}\|_{U^2_\Delta} \lesssim \log \mu \| \phi^{(3)}_{\mu}\|_{V^2_\Delta}
\]
and use \eqref{quadriU}.

For the second component we estimate directly the quadrilinear form by
\[
\begin{split}
|I_{3c}| \lesssim &\
\| P_\alpha H^{-1} ( \overline{ P_e Q_{\lesssim \alpha \lambda} \phi}_{\lambda} \partial_1
P_e Q_{\lesssim \alpha \lambda}  \phi_\lambda)\|_{L^\infty}
\| Q_{> \mu^N} \phi^{(3)}_{\mu}\|_{L^2} \| \partial_x  \phi^{(4)}_{\mu}\|_{L^2}
\\
\lesssim &\ \alpha \|  \overline{ P_e Q_{\lesssim \alpha \lambda} \phi}_{\lambda} \partial_1
P_e Q_{\lesssim \alpha \lambda}  \phi_\lambda)\|_{L^{\infty,1}_\e}
\mu^{-\frac{N}2} \| \phi^{(3)}_{\mu}\|_{V^2_\Delta} \mu \|  \phi^{(4)}_{\mu}\|_{V^2_\Delta}
\\
\lesssim &\ \alpha \mu^{1 -\frac{N}2}
\|\phi^{(1)}_{\lambda}\|_{\due} \|\phi^{(2)}_{\lambda}\|_{\due} \|\phi^{(3)}_{\mu}\|_{ l^2 V^{2,\sharp}}
  \| \phi^{(4)}_{\mu}\|_{l^2 V^{2,\sharp}}
\end{split}
\]

{\bf Case 3c(i), Step 4:} {\em Proof of \eqref{quadri}, conclusion.}
By the interpolation result in Lemma~\ref{lem:interpol} the estimate  \eqref{quadri}
follows from the bound \eqref{quadriUV} and the following estimate:
\[
|I_{3c}| \lesssim \alpha \mu \|\phi^{(1)}_{\lambda}\|_{|D|^{-\frac{1}{2}}U^p_\e}
\|\phi^{(2)}_{\lambda}\|_{|D|^{-\frac{1}{2}}U^{p}_\e} \|\phi^{(3)}_{\mu}\|_{ l^2
V^{2,\sharp}}
  \| \phi^{(4)}_{\mu}\|_{l^2 V^{2,\sharp}}
\]
where $p>2$.

This in turn is obtained as in the immediately preceding
computation but without any high modulation localization.

{\bf Case 3c(ii).} {\em Low ($ \lesssim \alpha \lambda$) modulation in
  both $H^{-1}$ and the frequency $\lambda$ factors, $\mu \approx
  \lambda$.}  Localizing to small angles in all four frequency $\lambda$
factors, here we need to consider the expression
\[
\begin{split}
I_{3c} =  & \ \int
P_\alpha \qo_{\lesssim \alpha \lambda}
  H^{-1}( \overline{ P_\e Q_{\lesssim \alpha \lambda} \phi}_{\lambda} \partial_1
P_\e Q_{\lesssim \alpha \lambda}  \phi_\lambda)
(\overline{ P_{\tilde \e} Q_{\lesssim \alpha \lambda}\phi}_{\mu} \partial_2
P_{\tilde \e} Q_{\lesssim \alpha \lambda}\phi_{\mu})  \ dx dt
\\ & \ - \int
P_\alpha \qo_{\lesssim \alpha \lambda}
  H^{-1}( \overline{ P_\e Q_{\lesssim \alpha \lambda} \phi}_{\lambda} \partial_2
P_\e Q_{\lesssim \alpha \lambda}  \phi_\lambda)
( \overline{ P_{\tilde \e} Q_{\lesssim \alpha \lambda} \phi}_{\mu} \partial_1
P_{\tilde \e} Q_{\lesssim \alpha \lambda}  \phi_{\mu}) \ dx dt
\end{split}
\]
We will prove that
\begin{equation}\label{quadri=}
|I_{3c}| \lesssim  (\log \mu)^5
\|\phi^{(1)}_{\lambda}\|_{V^{2,\#}} \|\phi^{(2)}_{\lambda}\|_{V^{2,\#}} \|\phi^{(3)}_{\mu}\|_{l^2 V^{2,\#}}
  \| \phi^{(4)}_{\mu}\|_{l^2 V^{2,\#}}
\end{equation}
in several steps:

\medskip

{\bf Case 3c(ii), Step 1:} {\em Proof of \eqref{quadri} for free waves, log loss.}
As before,   it suffices to estimate
\begin{equation}\label{4lin-rec=}
\left\lvert
\int
H^{-1}(\bar{\phi}_{R_1} \partial_1 \phi_{R_2})
\bar{\phi}_{R_3} \partial_2 \phi_{R_4} \ dx dt
-
\int
H^{-1}(\bar{\phi}_{R_1} \partial_2 \phi_{R_2})
\bar{\phi}_{R_3} \partial_1 \phi_{R_4} \ dx dt  \right\rvert
\end{equation}
where $R_1, R_2$ are $\alpha$-separated at frequency $\lambda$
and $R_3, R_4$ are $\alpha$-separated at frequency $\mu$.
We now split into two subcases according to the strength of the null form.

\medskip

\textbf{Subcase I:}  If there is $\beta \gtrsim \alpha \lambda$ such that
$| \xi_2 \wedge \xi_4 | \sim \beta$ for all $\xi_2 \in R_2$ and $\xi_4 \in R_4$
then we use the same argument as in  Case 3c(i).

\textbf{Subcase II:} Suppose now that $| \xi_2 \wedge \xi_4 | \lesssim \alpha \lambda$
for all $\xi_2 \in R_2$ and $\xi_4 \in R_4$ and also that $\mu \approx \lambda$.
Now we write the expression to control in the form
\[
\alpha \lambda
\left\lvert \int
P_\alpha H^{-\frac12} (\bar{\phi}_{R_1}  \phi_{R_2})
P_\alpha \bar H^{-\frac12} (\bar{\phi}_{R_3}  \phi_{R_4})(x, t) \
dt dx
\right\rvert
\]
This is estimated by applying  (\ref{bilinearL2}) to each of the two bilinear factors; there is
a $\log \lambda$ loss from the summation with respect to the $H^{-1}$ modulation.

\medskip

{\bf Case 3c(ii), Step 2:} {\em Proof of \eqref{quadri} for $U^2$ waves, log loss.}
Here we will show that
\begin{equation}\label{quadriU=}
|I_{3c}| \lesssim   \log \lambda
\|P_\e \phi^{(1)}_{\lambda}\|_{\due} \|P_\e \phi^{(2)}_{\lambda}\|_{\due} \|P_{\tilde \e} \phi^{(3)}_{\mu}\|_{ |D|^{-\frac{1}{2}}U^{2}_{\tilde \e}}
  \|P_{\tilde \e} \phi^{(4)}_{\mu}\|_{|D|^{-\frac{1}{2}}U^{2}_{\tilde \e}}, \qquad \mu \approx \lambda
\end{equation}
As in the similar argument in Case 3c(i) the problem reduces to the
case where each factor is localized in cubes of size $\alpha \time
\alpha \times \alpha \lambda$ located near the parabola. From here on,
the proof proceeds exactly as in the free case, using the
observation that the bilinear $L^2$ bounds
\eqref{BilinearEstimate} and  \eqref{bilinearL2}  extend to $U^2_\e$ functions.
This is somewhat less obvious for  \eqref{bilinearL2}; what helps is that
the $\alpha \times \alpha$ frequency localization allows for separation of
variables in $P_\alpha H^{-1}$, reducing the problem to a purely
temporal multiplier. But purely temporal multipliers interact well with the lateral $U^2$
atomic structure.

\medskip

{\bf Case 3c(ii), Step 3:} {\em Proof of \eqref{quadri=} for
$V^2$ waves, log loss.} Precisely, we will show that
\begin{equation}\label{quadriUV=}
|I_{3c}| \lesssim
(\log \mu)^5
\|\phi^{(1)}_{\lambda}\|_{l^2 V^{2,\sharp}} \|\phi^{(2)}_{\lambda}\|_{l^2 V^{2,\sharp}} \|\phi^{(3)}_{\mu}\|_{ l^2 V^{2,\sharp}}
  \| \phi^{(4)}_{\mu}\|_{l^2 V^{2,\sharp}}, \qquad \mu \approx\lambda
\end{equation}
The role of the $l^2$ structure here is to allow localization to the unit time scale.
Once this is done, from the fact that the $U^2_\Delta$ and $V^2_\Delta$ norms are equivalent at fixed modulation (see (\ref{uvbesov})), we have
\[
\|P_\e \phi_{\lambda}\|_{\due} \lesssim \log \lambda \| \phi_{\lambda}\|_{V^2_\Delta}
\]
Therefore \eqref{quadriUV=} follows from  \eqref{quadriU=}.

%%%%%%%%%%%%%%%%%%%%%%%%%%%%%%%%%%%%%%%%%%%%%%%%%%%%%%%
%%%%%%%%%%                                                     
%%%%%%%%%%                      Quintic terms                     
%%%%%%%%%%                                                     
%%%%%%%%%%%%%%%%%%%%%%%%%%%%%%%%%%%%%%%%%%%%%%%%%%%%%%%
\section{Quintic terms}

In this section, we focus on controlling the quintic terms
$N_{5,1}$, $N_{5,2}$, and $N_{5,3}$ of the nonlinearity, defined respectively by
\eqref{N51}, \eqref{N52}, and \eqref{N53}.
By pairing with a wave $\bar{\phi}$ and using duality as in the case of the trilinear terms,
controlling these terms is equivalent to controlling the integral
\[
I^6 = \int w_1 w_2 w_3  \ dx dt
\]
with
\[
\begin{split}
w_1 =  \ H^{-1} (\bar \phi^{(1)}_{\lambda_1} \partial \phi^{(2)}_{\lambda_2}),
\qquad
w_2 =  \ H^{-1} (\bar \phi^{(3)}_{\lambda_3} \partial\phi^{(4)}_{\lambda_4}),
\qquad
w_3 = \ \phi^{(5)}_{\lambda_5} \bar \phi^{(6)}_{\lambda_6} % \label{sextilinear}
\end{split}
\]
and its variations obtained by moving the derivative from one factor to the other
in each pair and by replacing $H$ by $\bar H$.

We denote by $\lambda$ the largest of the $\lambda_j$'s (which must
appear at least twice) and by $\lambda_0$ the smallest. We also
assume the normalization
\[
\|\phi^{(j)}_{\lambda_j}\|_{l^2 V^{2,\sharp}} = 1
\]
Then we need to establish the estimate
\begin{equation}\label{todo-6}
|I^6| \lesssim \lambda_0^s, \qquad s > 0
\end{equation}

Adding frequency localizations to each of the bilinear expressions, we can
replace $w_1,w_2,w_3$ by
\begin{equation}
w_1 = P_{\mu_1} H^{-1} (\bar \phi^{(1)}_{\lambda_1} \partial \phi^{(2)}_{\lambda_2}) ,
\qquad w_2=
 P_{\mu_2} H^{-1} (\bar \phi^{(3)}_{\lambda_3} \partial
\phi^{(4)}_{\lambda_4}),
\qquad w_3 =  P_{\mu_3}(\phi^{(5)}_{\lambda_5} \bar \phi^{(6)}_{\lambda_6})
\end{equation}
The $\mu_j$ summation is straighforward since we must have either $\mu_j = \lambda$
or $\mu_j \leq \lambda_0$. We can harmlessly  order the frequencies
as follows:
\[
\lambda_1 \leq \lambda_2, \qquad \lambda_3 \leq \lambda_4,  \qquad
\lambda_5 \leq \lambda_6
\]
Each of the three bilinear expressions is called unbalanced if $\mu_j \approx \lambda_{2j}$
and balaced otherwise (i.e. if $\mu_j \ll \lambda_{2j-1} \approx \lambda_{2j}$).
We begin by dispensing with the easy cases:

% The possible inputs are summarized as follows, in roughly increasing
% order of difficulty.

% \begin{tabular}{ccc}
% \hline
% $\lambda_1, \lambda_2$ & $\lambda_3, \lambda_4$ & $\lambda_5 , \lambda_6$ \\
% \hline
% unbalanced & unbalanced & either \\
% unbalanced & balanced & either \\
% balanced & balanced & unbalanced \\
% balanced & balanced & balanced \\
% \hline
% \end{tabular}

\textbf{Case 1: Unbalanced-unbalanced-either}
Lemma \ref{L:AxAngStrichartz}, the Sobolev embedding (\ref{Sobem2}),
and the $L^4$ Strichartz estimate (\ref{est:strichartz})
allow us to place each of the first two bilinear factors in $L^4$ and the
third one in $L^2$.

\textbf{Case 2: Unbalanced-balanced-either} Here we put the term identified as balanced
in $L^{\infty, 3}_{\e}$, the one identified as unbalanced in
$L^{2,6}_{\e}$, and the remaining $\phi^2$ term in $L^2$. This we can
achieve thanks to Lemma \ref{L:AxAngStrichartz}, the Sobolev embedding
(\ref{Sobem3}), and the $(q,r) = (4,4)$ Strichartz estimate
(\ref{est:strichartz}).

\textbf{Case 3: Balanced-balanced-either}
This case is also easy if high modulations are present in one of the balanced couples,
say the first one. Indeed, if one of the modulations there is much larger than $\mu_1 \lambda_1$, then we can use estimate \eqref{highmod-bi} to conclude as in Case 2.
Thus we are left with the low modulation case, where after some relabeling
we need to consider $w_1,w_2,w_3$
of the form
\[
\begin{split}
 w_1 = & \ P_{\mu_1}  H^{-1} (\overline {Q_{\lesssim \mu_1 \lambda_1}\phi}_{\lambda_1} \partial  Q_{\lesssim \mu_1 \lambda_2} \phi_{\lambda_1}),
\\
w_2 = & \ P_{\mu_2} H^{-1} (\overline{Q_{\lesssim \mu_2 \lambda_2} \phi}_{\lambda_2} \partial
Q_{\lesssim \mu_2 \lambda_2} \phi_{\lambda_2})
\\
w_3 = & \ P_{\mu_3} (\phi_{\lambda_3} \bar \phi_{\lambda_4})
\end{split}
\]
By the Littlewood-Paley trichotomy, the three output frequencies
$\mu_1, \mu_2$ and $\mu_3$ are not independent.  We denote the larger
magnitude scale, which is shared by two frequencies, by $\mhi$, and
the smaller one by $\mlo$.

We begin with a simpler computation, which applies under the assumption that
$\lambda_1 \not \in \{ \lambda_3,\lambda_4\}$.  Under this condition we claim that
\begin{equation}\label{cross-bi}
\| w_1 w_3\|_{L^1} \lesssim \log^2 ( \min \{\lambda_1,\lambda_4\})
\frac{\lambda_1}{\mu_1 (\lambda_1+\lambda_3)^\frac12
(\lambda_1+ \lambda_4)^\frac12}
\end{equation}
To see this we begin by harmlessly inserting angular localizations
$P_\e$ in the two factors in $w_1$.  Making the stronger assumption
that $P_\e Q_{\lesssim \mu_1 \lambda_1}\phi_{\lambda_1}$ is in either
$U^2_\Delta$ or $U^2_\e$ we can localize both factors further to
nearby frequency cubes $R_1,R_2$ of size $\mu \times \mu \times \lambda \mu$ using the
square summability provided by Lemma~\ref{U2:sqsum}.  Using the kernels $K_{\mu_1}$
and $K_{\mu_3}^0$  of $P_{\mu_1}H^{-1}$, respectively $P_{\mu_3}$, we can write
\[
\begin{split}
w_1 w_3(t,x) =
\int & \ K_{\mu_1}(s,y)(\overline P_{R_1}{Q_{\lesssim \mu_1 \lambda_1}\phi}_{\lambda_1}
\partial  P_{R_2} Q_{\lesssim \mu_1 \lambda_2} \phi_{\lambda_1})(t-s,x-y) \\ & \
K_{\mu_3}^0(y_1)  (\phi_{\lambda_3} \bar \phi_{\lambda_4})(t,x-y_1) \ ds dy dy_1
\end{split}
\]
Then we match one $\phi_{\lambda_1}$ factor with $\phi_{\lambda_3}$ and one with
$\phi_{\lambda_4}$ and apply twice the estimate \eqref{eq:bilinearUV1}, also noting that
$\| K_{\mu_1}\|_{L^1} \lesssim \mu_1^{-2}$ and $\|K_{\mu_3}^0\|_{L^1} \lesssim 1$.
This  gives the bound \eqref{cross-bi} but using $U^2$ type norms for the
$\phi_{\lambda_1}$ factors. The transition to $V^2$ norms is made trivially
if $\lambda_1 \lesssim  \lambda_4$ (which allows $\log \lambda_1$ losses).
Else we use the $U^2_\e$ norms and transition to $V^2_\e$ norms via
Lemma~\ref{lem:interpol}.

Then by using \eqref{cross-bi} for
$w_1 w_3$ and the $L^\infty$ bound \eqref{Hinftybound} for $w_2$ we obtain
\begin{equation}\label{6-cross}
|I^6| \lesssim (\log \lambda_0)^2 \lambda_1 \lambda_2 \mu_1^{-2}
\frac{\mu_1}{(\lambda_1+\lambda_3)^\frac12
(\lambda_1+ \lambda_4)^\frac12} \frac{\mu_2}{\lambda_2} =
 (\log \lambda_0)^2  \frac{  \lambda_1}{(\lambda_1+\lambda_3)^\frac12
(\lambda_1+ \lambda_4)^\frac12} \frac{\mu_2}{\mu_1}
\end{equation}
This will be used to dispense with some of the easier cases in the sequel.

{\bf Case 3(a): Balanced-balanced-unbalanced.}
This case is fully handled via \eqref{6-cross}. To see that we note that in this case
we have $\mu_3 = \lambda_3$.

If $\mu_3 = \mu_{lo}$ then we must have  $\mu_1= \mu_2 \geq \lambda_3$ which implies
that $\lambda_1,\lambda_2 \gg \lambda_3$. Thus \eqref{6-cross} applies and suffices.

If $\mu_3 = \mu_{hi}$ then we can assume that $\mu_1 = \mu_3 \gg \mu_2$. Hence
$\lambda_1 \gg \lambda_3$, and \eqref{6-cross} again applies and suffices.

{\bf Case 3(b): Balanced-balanced-balanced.} Here we have $\lambda_4 = \lambda_3
\gg \mu_3$. We first use \eqref{6-cross} to reduce the number of cases.

If $\lambda_1 \neq \lambda_3$ and $\lambda_2 \neq \lambda_3$ then we
can assume that $\mu_2 \leq \mu_1$ and  \eqref{6-cross} is enough.

If $\lambda_1 \neq \lambda_3$ but  $\lambda_2 = \lambda_3$ then  \eqref{6-cross}
suffices only if $\mu_1 \gtrsim \mu_2$.

Thus we are left with two cases:
\begin{itemize}
\item $\lambda_1 \neq \lambda_2 = \lambda_3$ and
$\mu_1 \ll \mu_2 = \mu_3 \ll \lambda_3$.
\item $\lambda_1 = \lambda_2 = \lambda_3$ and $\mu_1 \leq \mu_2$.
\end{itemize}
At this point the factors in $w_1$ and $w_2$ are restricted to low modulations,
but not those in $w_3$.  However  it is easy to  reduce the problem in both of these cases
 to small modulations ( $\leq \mu_3 \lambda_3$). Indeed, suppose that
one of the $\phi_{\lambda_3}$ has high modulation. Then we
use  \eqref{eq:Hinftybound} to bound $w_1$ in $\mu_1 L^\infty$.
For $w_2$ we use \eqref{L2biUV} and Bernstein to bound it in $\mu_2^{-1}
\lambda_2^{\frac12} L^2 L^4$ while  $w_3$ is in $\lambda_3^{-\frac12} L^2 L^{-\frac43}$,
again by Bernstein. Thus from here on we assume that
\[
w_3 = P_{\mu_3} (\overline {Q_{\lesssim \mu_3
\lambda_3}\phi}_{\lambda_3}   Q_{\lesssim \mu_3 \lambda_3}
\phi_{\lambda_3})
\]

{\bf Case 3(b)(i):} $\lambda_1 \neq \lambda_2 = \lambda_3$ and
$\mu_1 \ll \mu_2 = \mu_3 \ll \lambda_3$. Even though  \eqref{6-cross} does not cover
this case in full, we can still get some use out of it if we restrict ourselves
to high modulations of  $w_1$, namely
\[
w_1^{hi} =  \ P_{\mu_1}  H^{-1} \qo_{\gtrsim \mu_2^2}(\overline {Q_{\lesssim \mu_1 \lambda_1}\phi}_{\lambda_1} \partial  Q_{\lesssim \mu_1 \lambda_2} \phi_{\lambda_1})
\]
Then the $L^1$ norm of the kernel for $P_{\mu_1}  H^{-1} \qo_{\gtrsim \mu_2^2}$ is $\mu_2^{-2}$,
which suffices.

It remains to consider the low modulations of $w_1$,
\[
w_1^{lo} = \ P_{\mu_1}  H^{-1} \qo_{\ll  \mu_2^2}(\overline {Q_{\lesssim \mu_1 \lambda_1}\phi}_{\lambda_1} \partial  Q_{\lesssim \mu_1 \lambda_2} \phi_{\lambda_1})
\]
But in this case the modulations of $w_2$ and $w_3$ are
comparable, say equal to $\nu > \mu_2^2$ (or both $\leq \nu =
\mu_2^2$). Then we apply the $L^\infty$ bound
\eqref{eq:Hinftybound} to  $w_1^{lo}$ and the $L^2$ bound
\eqref{L2biUV} to $\qo_\nu w_2$ and $\qo_\nu w_3$ to obtain
\[
|I^6| \lesssim \frac{\lambda_1 \lambda_2}{\nu}
\frac{\mu_1}{\lambda_1}
 \frac{\nu}{\mu_2 \lambda_2} \lesssim 1
\]

{\bf Case 3(b)(ii): $\lambda_1 = \lambda_2 = \lambda_3=\lambda.$}

In this case we introduce full modulation truncations for each of the three
bilinear factors and consider
\[
I^6 = \int \qo_{\nu_1} w_1 \qo_{\nu_2} w_2 \qo_{\nu_3} w_3 \ dx dt
\]
where
\begin{equation}
\mu_1^2 \leq \nu_1 \leq \mu_1 \lambda_1, \qquad \mu_2^2 \leq \nu_2 \leq \mu_2 \lambda_2,
\qquad  \mu_3^2 \leq \nu_3 \leq \mu_3 \lambda_3
\label{ModCon}
\end{equation}
The reason for the lower bounds is that the symbol of $H^{-1}$ no longer changes
at lower modulations. Implicitly we allow a slight abuse of notation, where
for $\nu_i= \mu_i^2$ we replace $\qo_{\nu_i}$ by $\qo_{\lesssim \mu_i^2}$.
The two largest modulations are comparable;
we use $\nhi$ to denote their scale, along with $\nlo$ to denote
the scale of the remaining modulation, $\nlo \lesssim \nhi$.

Finally, we remark that the modulation summation only yields
acceptable logarithmic losses; in particular the $U^2$ and $V^2$ norms are
logarithmically close.

{\bf Subcase I.} $\mu_{lo}$ is paired with $\nu_{lo}$.
Then we apply the bilinear $L^\infty$ estimate \eqref{eq:Hinftybound} for the corresponding factor
and the bilinear $L^2$ bound \eqref{L2biUV} for the remaining factors to obtain
\[
|I^6| \lesssim (\log \lambda_0)^{4}
 \frac{\lambda^2}{\nu_1 \nu_2} \frac{\nu_{lo} \mu_{lo}}{\lambda}
\frac{ \nu_{hi}}{\mu_{hi} \lambda} \lesssim  (\log \lambda_0)^{4} \frac{\mu_{lo}}{ \mu_{hi}}
\]

{\bf Subcase II.} $\nu_3 = \nu_{lo}$, $\mu_1 = \mu_{lo}$. We apply
the bilinear $L^\infty$ estimate \eqref{eq:Hinftybound} for $P_{\nu_1} w_1$ and the $L^2$
 bound \eqref{L2biUV} for the remaining factors and conclude as above.

{\bf Subcase III.} $\nu_1 = \nu_{lo}$, $\mu_1 = \mu_{hi}$. Suppose that $\mu_2 = \mu_{lo}$,
as the argument is similar in the other case. This is the most difficult case.
We begin with several reductions.

We first localize each pair of factors to small angles using
multipliers $P_{\e_1}$, $P_{\e_2}$ and $P_{\e_3}$. Since $\log
\lambda$ losses are allowed in this case, it suffices to work with
factors in the spaces $X^{0,\frac12,1}$. Further, by orthogonality we
can reduce the problem to the case when both $\phi_{\lambda_j}$ factors are
frequency localized to cubes of size $\mu_j \times \mu_j \times
\nu_j$.  Then $w_j$ has a similar localization, and the $L^2$ bound
given by \eqref{L2biUV}.

If both $\phi_{\lambda_j}$ were free waves, then in effect $w_j$
would be localized in  smaller regions, namely a tilted cube $R_j$
of size $\mu_j \times \frac{\mu_j^2}{\lambda} \times \nu_j$.  Then
we can estimate
\[
\begin{split}
\| w_1 w_2\|_{L^2} \lesssim & \
\|w_1\|_{L^2} \|w_2\|_{L^2} \sup_{(\tau,\xi)} |R_1 \cap (\tau,\xi) - R_2|
\\
\lesssim & \  \frac{\lambda^2}{\nu_{lo}\nu_{hi}}
\left(\frac{\nu_{lo}}{\mu_{hi} \lambda}\right)^\frac12
\left(\frac{\nu_{hi}}{\mu_{lo} \lambda}\right)^\frac12 \left(\nu_{lo} \mu_{lo} \frac{\mu_{lo}^2}{\lambda}\right)^\frac12
\\
\lesssim & \ \left(\frac{\lambda \mu_{hi}}{\nu_{hi}}\right)^\frac12
\left(\frac{\mu_{lo}}{\mu_{hi}}\right) 
\end{split}
\]
Combined with the $L^2$ bound for $w_3$, this leads to the desired conclusion.

In order to gain a similar localization in the case when the factors are $X^{0,\frac12,1}$
functions, we foliate them with respect to the modulation, with integrability
with respect to the modulation parameter. For pointwise fixed modulation parameters
in all six factors the above argument still applies, and the conclusion follows.

%%%%%%%%%%%%%%%%%%%%%%%%%%%%%%%%%%%%%%%%%%%%%%%%%%%%%%%
%%%%%%%%%%                                                  
%%%%%%%%%%                      Conclusion                    
%%%%%%%%%%                                                      
%%%%%%%%%%%%%%%%%%%%%%%%%%%%%%%%%%%%%%%%%%%%%%%%%%%%%%%
\section{Conclusion}

\label{s:last}

It remains to show that the error terms may be controlled.
Recall from \S \ref{reductionEq} that we have
\[
\begin{split}
E_1 &= H^{-1}( H^{-1} A_x(0) |\phi|^2) \partial \phi  \\
E_2 &= H^{-1}\partial ( H^{-1} A_x(0) |\phi|^2)  \phi  \\
E_3 &= H^{-1} A_x(0) H^{-1}( \bar \phi \partial \phi) \phi  \\
E_4 &=(H^{-1} A_x(0))^2  \phi +   H^{-1} A_x(0) B   \phi + B^2 \phi \\
E_5 &= H^{-1}(B |\phi|^2) \partial \phi  \\
E_6 &= H^{-1}\partial ( B |\phi|^2)  \phi  \\
E_7 &= H^{-1}( \bar \phi \partial \phi) B \phi
\end{split}
\]
Some of these terms are related via duality.
In particular, estimating $E_1$ paired with $\phi$ is equivalent to estimating
\[
H^{-1}A_x(0) |\phi|^2 H^{-1}(\phi \partial \phi)
\]
In fact control on this term also gives control on $E_2$ paired with $\phi$.
Similarly, estimating $E_5$ paired with $\phi$ is equivalent to estimating
\[
B|\phi|^2H^{-1}(\phi \partial \phi)
\]
Control on this term also gives control on $E_6$ paired with $\phi$.

To obtain estimates, we use
\eqref{e:aFlow} for $H^{-1}A(0)$, which
gives $H^{-1}A(0) \in L^4_{t,x}[0, 1] \cap L^{2,6}_\e[0, 1]$;
\eqref{e:B} for $B = H^{-1}(A|\phi|^2)$, which
provides $B \in H^{-\frac12}L^2_{t,x}[0, 1] \in L^4_{t,x}[0, 1] \cap L^{2,6}_\e[0, 1]$;
Strichartz for $\phi$, providing $\phi \in L^4_{t,x}$;
and finally the $L^{\infty,3}_\e$ or $H^{-\frac12}L^2_{t,x}$ bounds on
$H^{-1}(\phi \partial \phi)$ coming from Lemma \ref{L:AxAngStrichartz}.

In fact, the estimates on $H^{-1}A(0)$, $B$, $\phi$, and $H^{-1}(\phi \partial \phi)$
come with extra regularity. This extra regularity guarantees all of the frequency summations.

%%%%%%%%%%%%%%%%%%%%%%%%%%%%%%%%%%%%%%%%%%%%%%%%%%%%%%%
%%%%%%%%%%                                                  
%%%%%%%%%%                      References                      
%%%%%%%%%%                                                     
%%%%%%%%%%%%%%%%%%%%%%%%%%%%%%%%%%%%%%%%%%%%%%%%%%%%%%%
\bibliography{CSS-bib}
\bibliographystyle{amsplain}

\end{document}